\documentclass{amsart}

\usepackage{geometry}
\usepackage{enumerate}
\usepackage{enumitem}
\usepackage{amsfonts}
\usepackage{amssymb}
\usepackage{xy}
\usepackage{sseq}
\input xy
\xyoption{all}
\usepackage{amsmath}
\usepackage{url}
\usepackage{verbatim}
\usepackage{hyperref}

\usepackage{xy}
\input xy
\xyoption{all}

\newcommand{\R}{\mathbf{R}}
\newcommand{\psp}{\spec^{\circ}}
\newcommand{\grspec}{\mathrm{GrSpec}}
\newcommand{\coker}{\mathrm{coker\, }}
\newcommand{\loc}{\mathrm{Loc}}
\newcommand{\spec}{\mathrm{Spec}}
\newcommand{\idem}{\mathrm{Idem}}

\newcommand{\einf}{\mathrm{CAlg}}

\newcommand{\cl}{\colon}
\newcommand{\clg}{\mathrm{CAlg}}
\newcommand{\sym}{\mathrm{Sym}}
\usepackage{amsthm}
\usepackage{tikz}
\usepackage{cleveref/cleveref}

\renewcommand{\rightrightarrows}{\begin{smallmatrix} \to \\
\to \end{smallmatrix} }
\newcommand{\triplearrows}{\begin{smallmatrix} \to \\ \to \\ 
\to \end{smallmatrix} }
 
\newcommand{\pics}{\mathfrak{pic}}

\newcommand{\TMF}{\mathrm{TMF}}
\renewcommand{\ell}{\mathrm{Ell}}

\newcommand{\e}[1]{\mathbf{E}_{#1}}
\renewcommand{\hom}{\mathrm{Hom}}
\newcommand{\supp}{\mathrm{Supp}}
\newcommand{\qcoh}{\mathrm{QCoh}}
\newcommand{\md}{\mathrm{Mod}}

\newtheorem{theorem}{Theorem}[section]

\newtheorem{lemma}[theorem]{Lemma}
\newtheorem{proposition}[theorem]{Proposition}
\newtheorem{corollary}[theorem]{Corollary}

\theoremstyle{definition}

\newtheorem{remark}[theorem]{Remark}
\newtheorem{example}[theorem]{Example}
\newtheorem{definition}[theorem]{Definition}
\newtheorem{cons}[theorem]{Construction}

\begin{document}

\title{Residue fields for a class of rational
$\e{\infty}$-rings and applications}
\date{\today}
\author{Akhil Mathew}
\address{Harvard University, Cambridge, MA 02138}
\email{amathew@math.harvard.edu} 
\urladdr{http://math.harvard.edu/~amathew}

\begin{abstract}
Let $A$ be an $\e{\infty}$-ring over the rational numbers. If $A$
satisfies a noetherian condition on its homotopy groups $\pi_*(A)$, we
construct a collection of $\e{\infty}$-$A$-algebras that  realize on homotopy
the residue fields of $\pi_*(A)$. We prove an analog of the nilpotence theorem
for these residue fields. As a result, we are able to give a complete
algebraic description of the Galois theory of $A$ and of the thick
subcategories of perfect $A$-modules. We also obtain partial information on
the Picard group of $A$. 
\end{abstract}

\maketitle

\tableofcontents
\section{Introduction}
\subsection{Motivation}
The goal of this paper is to describe certain invariants of structured ring
spectra in characteristic zero. We start by first reviewing the motivation from
stable homotopy theory. 

The chromatic picture of stable homotopy theory identifies a class of
``residue fields'' which play an important role in global phenomena. 
Consider the following ring spectra: 
\begin{enumerate}
\item $H \mathbb{Q}$: rational homology.  
\item For each prime $p$, mod $p$ homology $H \mathbb{F}_p$. 
\item For each prime $p$ and height $n$, the $n$th Morava $K$-theory $K(n)$. 
\end{enumerate}

These all define multiplicative homology theories on the category of spectra
satisfying \emph{perfect} K\"unneth isomorphisms: they behave like fields.
Moreover, as a consequence of the deep nilpotence technology of \cite{DHS,
HS}, they are powerful enough to describe much of the structure of the stable
homotopy category. 
For example, one has the following result: 
\begin{theorem}[Hopkins-Smith \cite{HS}] \label{fieldsdetectnilp} Let $R$ be a ring spectrum and let
$\alpha \in \pi_*(R)$. Then $\alpha$ is nilpotent if and only if the Hurewicz
image of $\alpha $ in $\pi_*(F \otimes R)$ is nilpotent, as $F$ ranges over all
the ring spectra above. \label{nilpotence}
\end{theorem} 

This fundamental result was used in \cite{HS} to classify the \emph{thick subcategories} of the
category of finite $p$-local spectra for a fixed prime $p$: all thick
subcategories are defined by vanishing conditions for the various residue
fields. 
One can attempt to ask such questions not only for spectra but for general
symmetric monoidal, stable $\infty$-categories, as 
Hovey, Palmieri, and Strickland have considered in \cite{axiomatic}; whenever one has an analog of
\Cref{nilpotence}, it is usually possible to prove results along these lines. 

For instance, let $A$ be an $\e{\infty}$-ring. Then one can try to study such questions in
the $\infty$-category $\md(A)$ of $A$-modules. If $\pi_*(A)$ is concentrated in
even degrees and is \emph{regular noetherian,} then it is possible to
construct residue fields, prove an analog of \Cref{nilpotence}, and obtain a purely algebraic description of the thick
subcategories of perfect $A$-modules. This has been observed independently by a
number of authors. For $\e{\infty}$-rings (such as the $\e{\infty}$-ring
$\TMF$ of periodic topological modular forms) which are ``built up''
appropriately from such nice $\e{\infty}$-rings, it is sometimes possible to
construct residue fields as well. We used this to classify thick subcategories
for perfect modules over $\e{\infty}$-rings such as $\TMF$ in \cite{thick_am}. 

\subsection{Statement of results}
In this paper, we will study such questions over the \emph{rational numbers}. 
Let $A$ be a rational $\e{\infty}$-ring  such that the even
homotopy groups $\pi_{\mathrm{even}}(A)$ form a
noetherian ring and such that the odd homotopy groups $\pi_{\mathrm{odd}}(A)$
form a finitely generated $\pi_{\mathrm{even}}(A)$-module. We will call such
rational $\e{\infty}$-rings \emph{noetherian.}

For the statement of our first result, we work with $\e{\infty}$-rings containing a unit in degree
two. In this case, we will produce, for every prime ideal $\mathfrak{p} \subset \pi_0(A)$, a ``residue field'' of $A$, which will
be an $\e{\infty}$-$A$-algebra whose homotopy groups form a graded field. 
\begin{theorem}[Existence of residue fields] 
\label{existencerf} Let $A$ be a rational, noetherian 
$\e{\infty}$-ring containing a unit in degree two. 
Given a prime ideal $\mathfrak{p} \subset \pi_0(A)$, there exists
an $\e{\infty}$-$A$-algebra $\kappa(\mathfrak{p})$ such that $\kappa(\mathfrak{p})$ is
even periodic
and the map 
$\pi_0(A) \to \pi_0 (\kappa(\mathfrak{p}))$ induces the reduction $\pi_0 A \to
\pi_0(A)_{\mathfrak{p}}/\mathfrak{p}\pi_0 (A )_{\mathfrak{p}}$. 
$\kappa(\mathfrak{p})$ is unique up to homotopy as an object of the $\infty$-category $\clg_{A/}$
of $\e{\infty}$-rings under $A$.
\end{theorem}

We will prove an analog of \Cref{nilpotence} in $\md(A)$ for these residue
fields.

\begin{theorem}[Nilpotence]
\label{ournilpthm}
Suppose $A$ is as above, and let $B$ be an $A$-ring spectrum; that is, an
algebra object in the homotopy category $\mathrm{Ho}( \md(A))$. Let $x \in
\pi_*(B)$. Then $x$ is nilpotent if and only if for every prime
ideal $\mathfrak{p}
 \subset \pi_0 (A)$, the image of $x$ in $\pi_*(B \otimes_A
 \kappa(\mathfrak{p}))$ is nilpotent. \end{theorem} 

The proof of \Cref{ournilpthm} uses entirely different (and much simpler)
techniques than \Cref{fieldsdetectnilp}. However, the conclusion is similar,
and we thus find $\md(A)$ as an interesting new example of an ``axiomatic stable
homotopy theory'' (\cite{axiomatic}) where many familiar techniques can be
applied. 

In particular, from \Cref{ournilpthm}, we will 
deduce a classification of thick subcategories of the
$\infty$-category $\md^\omega(A)$ of perfect $A$-modules, for $A$ rational
noetherian (not necessarily containing a unit in degree two). 
Let $\pi_{\mathrm{even}}(A) = \bigoplus_{i \in 2 \mathbb{Z}} \pi_i(A)$; this is
a graded ring, so $\spec \pi_{\mathrm{even}}(A)$ inherits a
$\mathbb{G}_m$-action. 

\begin{theorem}[Thick subcategory theorem]
\label{thicksubcat}
Let $A$ be a rational, noetherian $\e{\infty}$-ring. 
The thick subcategories of $\md^\omega(A)$ are in natural correspondence with
the subsets of the collection of homogeneous prime ideals of $
\pi_{\mathrm{even}}(A)$ which are closed under specialization  or, equivalently,
specialization-closed subsets of the topological space associated to the stack $( \spec \pi_{\mathrm{even}}(A))/\mathbb{G}_m$. 
\end{theorem} 

In particular, we determine the \emph{spectrum} in the sense of Balmer
\cite{Bal05} of $\md^\omega(A)$ as the topological space associated to the
stack $( \spec \pi_{\mathrm{even}}(A))/\mathbb{G}_m$. 
It shows that the spectrum is determined in terms of $\pi_*(A)$,
i.e., the map of \cite{Balmer2} is an isomorphism. 

We will then apply these ideas to the computation of Galois groups,
which we introduced
in \cite{galgp} as an extension of Rognes's work \cite{rognes}. 
The use of residue fields in Galois theory goes back to Baker-Richter's work in
\cite{BR2}, which studied the Galois groups of Morava $E$-theories at odd
primes. 
We will show that the Galois theory of a noetherian rational $\e{\infty}$-ring 
is ``almost'' entirely algebraic. (The ``almost'' comes from, e.g., the possibility of adjoining roots of periodicity generators in degrees $2n, n>
1$.) We prove: 

\begin{theorem} 
\label{ourgaloistheorem}
If $A$ is a noetherian rational $\e{\infty}$-ring, then the Galois group of $A$
is the \'etale fundamental group of  the stack $\left(\spec
\pi_{\mathrm{even}}(A)\right)/\mathbb{G}_m$. \end{theorem} 

We note that the Galois group depends on the choice of a basepoint is not
truly ``canonical''; however, there is a canonical equivalence of the
associated Galois theories, or equivalently of the Galois \emph{groupoids}.

Finally, we will study the Picard groups of noetherian rational
$\e{\infty}$-rings. Here our results are much less conclusive, but we prove: 

\begin{theorem} 
If $A$ is a noetherian rational $\e{\infty}$-ring, then the cokernel of
the natural map $\mathrm{Pic}( \pi_*(A)) \to \mathrm{Pic}(A)$ (see
Construction~\ref{algpiccons}) is a torsion-free abelian group. 
\end{theorem}

Usually, results such as \Cref{thicksubcat} and \Cref{ourgaloistheorem} are proved using
strong homological assumptions on $\pi_*(A)$, e.g., that it is a regular ring. 
We will be able to get away with much weaker hypotheses on $\pi_*(A)$ (i.e., 
nothing close to regularity) because, over characteristic zero,
$\e{\infty}$-rings are  simpler. They have a more
algebraic feel which gives one
a wider range of techniques, and they have been studied in detail starting
with Quillen's work on rational homotopy theory \cite{quillenrat}. 
In particular, there are two basic coincidences that will be used in this paper. 
\begin{enumerate}
\item The free $\e{\infty}$-ring on a generator in degree zero is equivalent
to the suspension spectrum $\Sigma^\infty_+ \mathbb{Z}_{\geq 0}$. 
In particular, as a result, it is possible to quotient an $\e{\infty}$-ring by an element in
degree zero to get a new $\e{\infty}$-ring.
\item The free $\e{\infty}$-ring on a generator in degree $-1$ is equivalent to 
cochains on $S^1$. 
This has important descent-theoretic consequences and enables one to compare
modules over this $\e{\infty}$-ring with local systems on the circle $S^1$.
\end{enumerate}

Both these conditions are  specific to the rational numbers. 
They fail away from characteristic zero, because of
the existence of {power operations} \cite{homoiter, Hinfty}.

The above theorems rely crucially on the noetherianness hypotheses. We will
discuss various counterexamples in \S 8. These counterexamples are related to
classical purity questions 
for Picard groups and \'etale fundamental groups for local rings.
As a consequence, we produce new examples of non-algebraic Galois extensions of ring
spectra in the sense of Rognes \cite{rognes}, which seem to be of interest in itself. 

\begin{theorem} 
Given $n > 1$, 
there exists a rational $\e{\infty}$-ring $A$ with $\pi_0(A) = 
\mathbb{C}[x^n, y^{n-1}x , y^{n-2}x^2, \dots, y^n]$, $\pi_i(A) = 0$ for $i >
0$, and such that the Galois group of $A$ is $\mathbb{Z}/n$. 
\end{theorem} 

As the \'etale fundamental group of 
$\spec \mathbb{C}[x^n, y^{n-1}x , y^{n-2}x^2, \dots, y^n]$ is trivial, this Galois group is not
algebraic. To our knowledge, this is the first non-algebraic example of a
Galois extension of \emph{rational} ring spectra to be observed. 
We will describe explicitly the $\mathbb{Z}/n$-Galois extension in \S 8.

\subsection*{Organization} This paper is organized as follows. In \S 2, we
analyze the operation of attaching a 1-cell in a rational $\e{\infty}$-ring in
detail. In \S 3, we do the same for the operation of attaching a 0-cell, via a
comparison between modules over the cochain $\e{\infty}$-ring $C^*(S^1 ;
\mathbb{Q})$ and local systems on $S^1$. 
By reducing to the case where one has a unit in degree $2$, we will be able to
obtain all our constructions by only attaching cells in these degrees. 
The main technical results (existence
of the residue fields and the nilpotence theorem) are proved in \S 4. \S 5, \S
6, and \S 7 contain 
the applications to thick subcategories, Galois groups, and Picard groups,
respectively. Finally, \S 8 discusses various non-noetherian counterexamples. 

\subsection*{Notation} In this paper, we will adopt the following notational
conventions. 
We will frequently identify abelian groups with their Eilenberg-MacLane
spectra without additional notation. 
The letters $R, S, T, ...$ will refer to ordinary (discrete)
rings. The subscript ${}_*$ will refer to a grading.
We will let
$\clg$ denote the $\infty$-category of $\e{\infty}$-rings. 
The letters $A, B, C$ will refer to $\e{\infty}$-rings.
Given an
$\e{\infty}$-ring $A$, we let $\md(A)$ denote the $\infty$-category of
$A$-modules and $\md^\omega(A) \subset \md(A)$ the full subcategory spanned by
the perfect $A$-modules.
If $X$ is a space and $A$ an $\e{\infty}$-ring, we let $C^*(X; A)$ denote the
cochain $\e{\infty}$-algebra on $X$ with values in $A$, often also denoted $A^X$ or
$\mathrm{Fun}(X_+, A)$.
Finally, if $A \in \clg$ is a \emph{rational} $\e{\infty}$-ring, we will write $A[t_2^{\pm 1}]$ for the free
$\e{\infty}$-$A$-algebra on an invertible degree two generator. 

Although we use the language of $\infty$-categories and higher algebra
\cite{higheralg} in this paper, 
we note that 
everything can be carried out in the world of model categories. 
The theory of $\mathbf{E}_\infty$-ring spectra and modules over them was
originally developed \cite{EKMM} 
in the world of model categories. In characteristic zero, the issue
simplifies further as one can work with commutative differential graded algebras.
\subsection*{Acknowledgments} I am grateful to Bhargav Bhatt and Jacob Lurie for helpful
discussions related to the subject of this paper, and the referee for many
detailed comments. The author was supported by the NSF Graduate Fellowship
under grant DGE-114415.

\section{Degree zero elements of rational $\e{\infty}$-rings}

In this section, we 
describe the first set of the basic characteristic zero techniques needed for this paper. 
In particular, we discuss  the ``coincidence''  of $\e{\infty}$-rings that 
$\Sigma^\infty_+ \mathbb{Z}_{\geq 0}$ is free on a degree zero class
and thus analyze the operation of attaching cells in degree one. The main
result of the section (\Cref{radicial2}) controls the behavior on homotopy
rings of attaching cells in degree one. We also prove a version
(\Cref{secresfld}) of the
classical result that a complete local ring with residue field of characteristic
zero contains a copy  of its residue field. 

Some of the more refined results
require the noetherianness hypothesis that will be crucial for most of the main
results of this paper. 

\begin{definition}
We say that a rational $\e{\infty}$-ring $A$ is \emph{noetherian} if:
\begin{enumerate}
\item  
The commutative ring
$\pi_{\mathrm{even}}(A)$ is
noetherian.
\item  The $\pi_{\mathrm{even}}(A)$-module $\pi_{\mathrm{odd}}(A)$ is finitely
generated. 
\end{enumerate}
\end{definition}

The ``noetherian'' hypothesis ensures that certain categorical constructions
one may perform on $A$ affect the homotopy groups of $A$ in a reasonable
manner and, as such, will be indispensable to this paper. 

\newtheorem{warning}[theorem]{Warning}
\begin{warning}
The noetherian condition is not a purely categorical one. For instance, a finitely
presented $\e{\infty}$-algebra over a noetherian $\e{\infty}$-ring (even
$\mathbb{Q}$) need not be noetherian, i.e., the analog of Hilbert's basis
theorem fails. See \Cref{bigcompact} for an example. 
\end{warning}

\subsection{Cofibers of degree zero elements}
Let $A$ be an $\e{\infty}$-ring and let $x \in \pi_k(A)$, defining a map of
$A$-modules $\Sigma^k A \stackrel{x}{\to} A$. 

\begin{definition}
We will write $A/x$ for the cofiber of
this map $x\colon \Sigma^kA \to A$. \end{definition}

One wants to think of $A/x$ as a homotopy-theoretic ``quotient'' of
$A$ by the ``ideal'' generated by $x$ and, as in algebra, turn this into an
$\e{\infty}$-ring under $A$. 
There is, in general, no reason to expect this to be possible (or canonical in
any way). 
\begin{example}
The sphere $S^0$ is the most basic example of an $\e{\infty}$-ring, but 
the Moore spectrum $S^0/2$ is not 
even a ring spectrum up to homotopy. 
\end{example}
The obstructions to multiplicative structures have been discussed, for example, in
\cite{oka, productsMU}. Some further obstructions to \emph{structured} multiplications,
via the theory of power operations, are discussed in \cite{MNN}. 

Suppose first that $k = 0$.
To understand the failure of such quotients to be ring spectra, 
recall how the quotient is constructed in classical commutative algebra. Let
$R$ be a (classical) commutative ring, and fix $x \in R$. The (classical)
quotient  $R/(x)$ is the pushout of the diagram of commutative rings
\[ \xymatrix{
\mathbb{Z}[t] \ar[d]^{x \mapsto t} \ar[r]^{t \mapsto 0} &  \mathbb{Z} \ar[d] \\
R \ar[r] &  R/(x)
}.\]
Here $\mathbb{Z}[t]$ is the free commutative ring on a generator $t$, and
forming the pushout $R/(x)$ as above amounts to setting $x = 0$. 

In homotopy theory, one can make a similar construction, which has been
discussed  in, for example, \cite{szymik}. 

\begin{definition} \label{modx0}
There is a free
$\e{\infty}$-ring on a single generator, denoted $S^0\left\{t\right\}$, whose
underlying spectrum is given by 
\[ S^0 \left\{t\right\} \simeq \bigoplus_{n \geq 0} \Sigma^\infty_+ B \Sigma_n,\]
as $\left\{\Sigma_n\right\}_{n \geq 0}$ ranges over the symmetric groups. 
Given an $\e{\infty}$-ring $A$ and an element $x \in \pi_0 A$, we
obtain a map (defined up to homotopy) of $\e{\infty}$-rings $S^0\left\{t\right\} \to A$ sending the
tautological class $t \in \pi_0 S^0\left\{t\right\}$ to $x$, by the universal property: the space of maps of $\e{\infty}$-rings
$S^0\left\{t\right\} \to A$ is precisely $\Omega^\infty A$. 
In particular, we can form a pushout square
in $\clg$,
\[ \xymatrix{
S^0\left\{t\right\} \ar[d]^{t \mapsto x}  \ar[r]^{t \mapsto 0} &  S^0 \ar[d] \\
A \ar[r] &  A'
},\]
where $A'$ is called the \emph{free $A$-algebra with $x =0$.} 
Following Szymik \cite{szymik}, we will write $A' = A//x$. 
 
Given
an $\e{\infty}$-$A$-algebra $A''$, the space\footnote{We write $\clg_{A/}$ for
the $\infty$-category of $\e{\infty}$-$A$-algebras and $\hom_{A/}(\cdot,
\cdot)$ for mapping spaces here.} $\hom_{A/}( A', A'')$ is the space
of nullhomotopies of $x$ in $A''$ (which is empty unless $x $ maps to zero in
$\pi_0 A''$, and in this case is $\Omega^{\infty + 1} A''$). 
\end{definition}

For future reference, it will be convenient to have the following more general definition.
\begin{definition} \label{modxgeneral}
For $X$ a spectrum, we write $\sym^*(X)$ for the free $\e{\infty}$-ring on
$X$, so that $\sym^*(X) \simeq \bigoplus_{n \geq 0} (X^{\otimes n})_{h
\Sigma_n}$. If $A$ is an $\e{\infty}$-ring and $x \in \pi_k A$ is an element,
we denote by $A//x$ the pushout
\[ \xymatrix{
\sym^*(S^k) \ar[d]^0 \ar[r]^x &  A \ar[d] \\
S^0 \ar[r] &  A//x
},\]
where the map $\sym^*(S^k ) \to A$ is determined by the map $x\colon S^k \to A$,
and where $\sym^*(S^k) \to S^0$ is determined by $0\colon S^k \to S^0$.
We will call $A//x$ the \emph{free $\e{\infty}$-$A$-algebra with $x = 0$.} 
Given a sequence of elements $x_1, \dots, x_n \in \pi_*(A)$, we will write
$A//(x_1, \dots, x_n)$ for the iterated quotient $\left(\dots(A//x_1)//x_2)\dots
// x_{n-1} )\dots) \right) //x_n$.
\end{definition}

We return to the case $k = 0$.
In
general, if $x \in \pi_0 A$ is fixed, then 
\Cref{modx0}
gives 
\[ A//x  \simeq A \otimes_{S^0\left\{t\right\}} S^0, \]
since pushouts of $\e{\infty}$-rings are relative tensor products. 
This is usually very different, as an $A$-module, from $A/x$. 
For example, the free $\e{\infty}$-ring with $p^n = 0$ is not 
$S^0/p^n$. From the ``chromatic'' point of view, it is actually invisible: its
$E_r$-localization  vanishes for each $r$, by the main result of \cite{MNN}. 

\begin{remark}
In fact, the $S^0\left\{t\right\}$-module $S^0$ is quite complicated, and is
not, for example, perfect. 
For instance, if we worked over $\mathbb{F}_2$ rather than $S^0$, then
$\mathbb{F}_2\left\{t\right\}$ has homotopy groups given by a polynomial ring
on the tautological class $t$ and certain admissible monomials  in the
Dyer-Lashof algebra applied to $t$ (\cite{homoiter}), so that $\mathbb{F}_2$ is an infinite
quotient of $\mathbb{F}_2\left\{t\right\}$ by a regular sequence of
polynomial generators. 
\end{remark}

However, there is another $\e{\infty}$-ring which is better behaved in this
regard, and which \emph{will} enable us to place $\e{\infty}$-structures on
quotients in certain cases. Recall that the $\e{\infty}$-ring $S^0\left\{t\right\}$ is obtained from the free
$\e{\infty}$-space on a single generator by applying $\Sigma^\infty_+$. This
$\e{\infty}$-space is  the free symmetric monoidal
category on one object: the groupoid of finite sets and isomorphisms between
them, or topologically $\bigsqcup_{n \geq 0} B \Sigma_n$. 

\begin{definition}
We can apply $\Sigma^\infty_+$ instead to  the
symmetric monoidal groupoid $\mathbb{Z}_{\geq 0}$, which 
has objects given by the natural numbers (under addition) and no nontrivial isomorphisms. 
The resulting $\e{\infty}$-ring $\Sigma^\infty_+ \mathbb{Z}_{\geq 0}$, the
``monoid algebra''  of the natural numbers (as studied, for example, in \cite{ABGHR}), 
will be written $S^0[t]$ since its homotopy groups actually are given by $(\pi_*
S^0)[t]$. More generally, we will write $A[t]$ for the $\e{\infty}$-ring $A \otimes \Sigma^\infty_+
\mathbb{Z}_{\geq 0}$, if $A$ is any $\e{\infty}$-ring. 
\end{definition}

\begin{cons}
\label{modxeinfinity}
Now let $A$ be an $\e{\infty}$-ring, and let $S^0\left\{t\right\} \to A$ be a
map classifying an element $x \in \pi_0(A)$. 
Suppose that we have a factorization in the $\infty$-category $\einf$
\[ \xymatrix{
S^0\left\{t\right\} \ar[d]  \ar[r]^x &  A \\
S^0[t] \ar@{-->}[ru]
}.\]
In this case, we can form the relative tensor product  $A \otimes_{S^0[t]} S^0
$ (using the map $S^0[t] \to S^0$ which sends $t \mapsto 0$), as an
$\e{\infty}$-ring. The cofiber
sequence of $S^0[t]$-modules
\[ S^0[t] \stackrel{t}{\to} S^0[t] \to S^0,  \]
shows that this relative tensor product, as 
an $A$-module spectrum, is actually $A/x$. 
\end{cons}

In other words, by the universal
property of the monoid algebra,
if there exists a factorization in the diagram
 of $\e{\infty}$-spaces
 \[ \xymatrix{
\bigsqcup_{n \geq 0} B \Sigma_n  \ar[d] \ar[r]^x & \Omega^\infty A \\
\mathbb{Z}_{\geq 0} \ar@{-->}[ru]
},\]
where $\Omega^\infty A$ is given the \emph{multiplicative} $\e{\infty}$-structure, then 
we can place a \emph{natural} $\e{\infty}$-structure on $A/x$.
\begin{remark}
Let $X$ be an $\e{\infty}$-space and let $x \in \pi_0 X$, classified by a map
of $\e{\infty}$-spaces $\bigsqcup_{n \geq 0} B \Sigma_n \to X$.  
If this map admits a factorization over $\mathbb{Z}_{\geq 0}$, then 
$x$ has been called by Lurie a ``strictly commutative'' element of $X$. 
Construction~\ref{modxeinfinity} shows that, while arbitrary cofibers $A/x$ need not admit
$\e{\infty}$-structures, one can find an $\e{\infty}$-structure if $x$ is
strictly commutative. 
\end{remark}

Unfortunately, in general, describing maps out of $S^0[t]$ is difficult, since
$\mathbb{Z}_{\geq 0}$ does not admit a simple presentation as an
$\e{\infty}$-space.  In \emph{characteristic zero}, i.e., over
$\mathbb{Q}$, the natural map
\[ \mathbb{Q}\left\{t\right\}  \to \mathbb{Q}[t], \]
becomes an equivalence of $\e{\infty}$-rings, because the
maps $(B \Sigma_n)_+ \to S^0$ are rational equivalences. 
In particular, given any rational $\e{\infty}$-ring $A$, and an element $x \in
\pi_0(A)$, we can obtain a map
\[ \mathbb{Q}[t] \to A, \quad t \mapsto x,  \]
and we can form the relative tensor product $A/x \simeq A
\otimes_{\mathbb{Q}[t]} \mathbb{Q}$ as an $\e{\infty}$-ring. 
This process can equivalently be described as attaching a 1-cell to kill the element $x \in
\pi_0 A$, i.e., as forming $A//x$. 
We may summarize these observations in the following proposition.

\begin{proposition} 
\label{modxmodx}
Let $A$ be a rational $\e{\infty}$-ring and let $x \in \pi_0(A)$. Then  $A//x
\in \clg_{A/}$
 has as underlying $A$-module $A/x$.
In particular, we may make $A/x$ into an $\e{\infty}$-$A$-algebra. 
\end{proposition}


\label{sec:quotient0}

It is similarly possible to quotient by \emph{even
degree} elements of a rational $\e{\infty}$-ring, as we show below.  In fact, we did not strictly need the
discussion of ``strict  commutativity'' for the present paper, but included it
for its intrinsic interest (as it  becomes more relevant away from
characteristic zero). 

\begin{proposition} 
\label{modxmodxgr}
Let $A$ be a rational $\e{\infty}$-ring and let $x \in \pi_n(A)$. Suppose $n$
is an even integer. Then $A//x \in \clg_{A/}$ has underlying $A$-module $A/x$.
\end{proposition} 
\begin{proof} 
Recall that $\pi_* \sym^*(\mathbb{Q}[n])$ is a polynomial ring on a class in
degree $n$. 
Thus, the result follows from $A//x \simeq A \otimes_{\sym^*(\mathbb{Q}[n])}
\mathbb{Q}$ and the cofiber sequence $\Sigma^n \sym^*\mathbb{Q}[n] \to
\sym^*\mathbb{Q}[n] \to \mathbb{Q}$ of $\sym^* \mathbb{Q}[n]$-modules. 
\end{proof} 

\subsection{The Cohen structure theorem}

Let $(R, \mathfrak{m})$ be a complete local noetherian ring with residue field $k$ of
characteristic zero.  In this case, a basic piece of the Cohen structure
theorem (see for instance \cite[Ch. 8]{eisenbud}) implies that $R$ contains a copy of its residue field: 

\begin{theorem}[Cohen] \label{cohen}
Hypotheses as above, the projection $R \to R/\mathfrak{m}  =  k$ admits a
section. 
\end{theorem} 

We refer to \cite[Theorem 60, \S 28.J]{matsumura} for a proof of a
more general result than \Cref{cohen}. 
It is closely related to the fact that, in characteristic zero, all field
extensions can be obtained as an inductive limit of smooth morphisms; this argument implies an analogous result in the
world of $\e{\infty}$-rings, and it is the purpose of this subsection to
describe that. 
In particular, we prove:

\begin{proposition} 
\label{secresfld}
Let $A$ be a noetherian, rational $\e{\infty}$-ring such that $\pi_0 A$ is a 
complete local ring with residue field $k$. Then there exists a morphism of
$\e{\infty}$-rings $k \to A$ such that on $\pi_0$, the composite map
\( k \to \pi_0 A \to k  \)
is the identity.
\end{proposition} 

\begin{proof} 
By \Cref{cohen}, there is a section $\phi\colon k \to \pi_0 A$ of the reduction map.
We want to realize this topologically. 
To start with, we obtain a  map $\mathbb{Q} \to A$ as $A$ is rational.
Let $\left\{t_\alpha\right\}_{\alpha \in \Gamma}$ be a transcendence basis of
$k/\mathbb{Q}$ (cf. \cite[Ch.~VIII, sec. 1]{Lang} for a textbook reference), so
that we have field extensions
\[ \mathbb{Q} \subset \mathbb{Q}(\left\{t_\alpha\right\}) \subset k,  \]
where the first extension is purely transcendental and the second extension is
algebraic. For each $\alpha \in \Gamma$, choose $u_\alpha \in \pi_0 A$ be
defined by 
$u_\alpha = \phi(t_\alpha)$, so that $u_\alpha$ 
projects to $t_\alpha $ in the residue field. We obtain a map of
$\e{\infty}$-rings 
\[ \bigotimes_{\Gamma} \mathbb{Q}[t_\alpha] \to A, \quad t_\alpha \mapsto
u_\alpha,  \]
where the left-hand-side is a free $\e{\infty}$-ring on $|\Gamma|$ variables
(i.e., a discrete polynomial ring on the $\left\{t_\alpha\right\}$). It
necessarily factors over the  localization $\mathbb{Q}(\left\{t_\alpha\right\})$, so we
obtain a map
$\mathbb{Q}(\left\{t_\alpha\right\}) \to A$.
This realizes on $\pi_0$ the restriction
$\phi|_{\mathbb{Q}(\left\{t_\alpha\right\})}$.

Finally, we want to find an extension over
$k$ such that the diagram
\[ \xymatrix{
\mathbb{Q}(\left\{t_\alpha\right\}) \ar[r]\ar[d]   &  k \ar@{-->}[ld] \\
A
},\]
such that the composite $k \to \pi_0 A \to k$ is the identity. Since $k$ is a
colimit of finite \emph{\'etale}
$\mathbb{Q}(\left\{t_\alpha\right\})$-algebras (i.e., finite separable
extensions; recall that we are in characteristic zero), it is equivalent to doing this at the level of $\pi_0$ (\cite[\S
7.5]{higheralg}), and 
the map $\phi\colon k \to \pi_0(A)$ enables us to do that. 
\end{proof} 

\begin{remark}
The argument shows that the set of \emph{homotopy classes} of
maps of $\e{\infty}$-rings $k \to A$ is in bijection with the set of
ring-homomorphisms $k \to \pi_0(A)$. 
\end{remark}

\subsection{Properties of the quotient $\e{\infty}$-ring}

We will need a few more
preliminaries on the construction of Definition~\ref{modx0} and its behavior on homotopy. If $A$ is a rational $\e{\infty}$-ring and $x \in \pi_0
A$, then the homotopy groups of $A//x \simeq A/x$ are determined additively by the
short exact sequence 
\begin{equation} \label{ses} 0 \to \pi_j(A)/x \pi_j(A) \to \pi_j(A/x) \to (\mathrm{ker\ } x)|_{\pi_{j-1}
(A)}
\to 0,\end{equation}
but this fails to determine the precise multiplicative structure. In this
section, we will show (\Cref{radicial2}) that the multiplicative structure is not so different
from that of the subring $\pi_*(A)/x \pi_*(A)$ under noetherian hypotheses. 
We will not need the full strength of these results in the sequel.

We begin by reviewing the theory of finite universal homeomorphisms \cite[Tag 04DC]{stacks-project}.
Recall that a map of rings $R \to R'$ is called a \emph{universal
homeomorphism} if, for every $R$-algebra $R''$, the map $R'' \to R'' \otimes_R
R'$ induces a homeomorphism upon applying $\mathrm{Spec}$.

\begin{proposition} 
A morphism $R \to R'$, such that $R'$ is a finitely generated $R$-module, is a
universal homeomorphism if and only if, for every morphism $R \to k$ where $k$
is a field, the base-change $R' \otimes_R k$ is a local artinian $k$-algebra
(in particular, nonzero). 
\end{proposition} 
\begin{proof} 
Suppose that for every map from $R$ to a field $k$, the base-change $R'
\otimes_R k$ is a local artinian $k$-algebra. It follows that the residue
field of $R' \otimes_R k$ is necessarily a purely inseparable extension of
$k$, since otherwise we can replace $k$ by $\overline{k}$ and $R' \otimes_R
\overline{k}$ would have nontrivial idempotents.
It then follows that the map $\mathrm{Spec} R' \to \mathrm{Spec} R$ is a radicial morphism 
 \cite[Tag 01S2]{stacks-project}.
Given any $R$-algebra $R''$, the map $R'' \to R' \otimes_R R''$ is finite, so
induces a closed map on $\mathrm{Spec}$ which is also injective since $R \to
R'$ is radicial. The map is surjective since all the fibers are nonempty by
assumption, and thus a homeomorphism. 

Conversely, suppose $R \to R'$ is a finite universal homeomorphism. Then all
the base-changes $k \to R' \otimes_R k$, for an $R$-field $k$, are universal
homeomorphisms themselves. 
In particular, $\mathrm{Spec} (R' \otimes_R k)$ is connected. But $R' \otimes_R k$ is a finite-dimensional
$k$-algebra, so if its spectrum is connected, then $R' \otimes_R k$ must be
local artinian. 
\end{proof} 

\begin{corollary}
A finite map $R \to R'$ of $\mathbb{Q}$-algebras is a universal homeomorphism if
 and only if for every residue field $R \to k$,
the tensor product $R'
\otimes_R k$ is local with residue field $k$. 
\end{corollary}
\begin{proof} 
This follows from the fact that all finite extensions in characteristic zero
are separable, 
so that if $A$ is a finite-dimensional local $k$-algebra with residue field 
strictly containing $k$, then $A \otimes_k \overline{k}$ necessarily has
nontrivial idempotents. 
\end{proof} 

We will now begin working towards the proof of \Cref{radicial2}. We will first
need a preliminary lemma on idempotents in these quotients. 

\begin{definition} 
If $A$ is an $\e{\infty}$-ring, we let $\mathrm{Idem}(A)$ denote the set of
idempotents in $\pi_0 A$. The construction $A \mapsto \mathrm{Idem}(A)$ sends
homotopy limits of $\e{\infty}$-rings to inverse limits of sets, since
the set $\mathrm{Idem}(A)$ is homotopy equivalent to the space of maps of
$\e{\infty}$-rings $S^0
\times S^0 \to A$ in view of the theory of \'etale algebras
\cite[\S 7.5]{higheralg}.  If $R$ is a discrete ring, we will also
write $\mathrm{Idem}(R)$ for the set of idempotents in $R$.
\end{definition}

\begin{lemma} \label{purelyinsepart}
Let $A$ be a rational noetherian $\e{\infty}$-ring with $\pi_0(A)$ local and 
let $x_1, \dots, x_r \in \pi_0 (A)$. Then the map $\pi_0 (A)/(x_1, \dots, x_r)
\to \pi_0(A//(x_1, \dots, x_r))$ of discrete rings induces an
isomorphism on $\mathrm{Idem}$.\end{lemma} 
\begin{proof} 
In fact, we have a map of $\e{\infty}$-rings
\( A \to A//(x_1, \dots, x_r),  \)
and if we form the cobar construction on this map,
we obtain
a coaugmented cosimplicial object
\[   A//(x_1, \dots, x_r) \rightrightarrows  A//(x_1, \dots, x_r) \otimes_A
A//(x_1, \dots, x_r) \triplearrows \dots, \]
whose homotopy limit is the $(x_1,\dots, x_r)$-adic completion of $A$. We
refer to 
\cite[\S 4]{DAGProp} for preliminaries on completions of ring spectra. In
particular, the idempotents in the totalization are the same as the idempotents
in the $(x_1, \dots, x_r)$-adic completion of $\pi_0 A$, or equivalently, by
the lifting idempotents theorem \cite[Cor. 7.5]{eisenbud}, in $\pi_0(A)/(x_1, \dots, x_r)$. 

Since the operation of taking idempotents commutes with homotopy limits, we
conclude that 
the set of idempotents in $\pi_0(A)/(x_1, \dots, x_r)$ is the reflexive
equalizer 
\[ \idem(A//(x_1, \dots, x_r)) \rightrightarrows \idem(A//(x_1, \dots, x_r)
\otimes_A A//(x_1, \dots, x_r)).  \]
However, we claim that the two maps in the reflexive equalizers are
isomorphisms (and thus equal).
In fact, 
\( A//(x_1, \dots, x_r) \otimes_A A//(x_1, \dots, x_r)  \)
is
obtained by attaching 1-cells to kill the classes $x_1, \dots, x_r$ in $A//(x_1, \dots,
x_r)$ which are \emph{already zero}; in particular, as an $\e{\infty}$-ring, we
have
$$ A//(x_1, \dots, x_r) \otimes_A A//(x_1, \dots, x_r)  \simeq A//(x_1, \dots,
x_r) \otimes \sym^*[y_1, \dots, y_r], \quad |y_i| = 1,$$
which has the same idempotents as 
$A//(x_1, \dots, x_r)$.
\end{proof}

\begin{proposition} 
\label{nilpextmod}
Let $A$ be a rational noetherian $\e{\infty}$-ring and let $x \in
\pi_0(A)$. 
The map $\pi_0(A)/(x) \to \pi_0(A//x)$ is a finite universal homeomorphism. 
\end{proposition} 

\begin{proof} 
We already know that $\pi_0(A//x)$ is a finitely generated
$\pi_0(A)/(x)$-module, by the short exact sequence \eqref{ses}. 
We will check that the map is a finite universal homeomorphism fiberwise at each
prime. 

Fix a prime ideal $\mathfrak{p}$ of $\pi_0(A)/(x)$.
Let $x_1, \dots, x_n \in  \pi_0(A)$ project to generators of $\mathfrak{p}$.
Localizing $A$ at
$\mathfrak{p}$, we may assume that $\pi_0(A)$ is local with maximal ideal
$\mathfrak{p}$. By completing $A$, we may assume that 
that $A$ admits the structure of an $\e{\infty}$-$k$-algebra for $k$ the
residue field of $\pi_0(A)$, in view of \Cref{secresfld}. 

We need to show that 
the map of commutative rings 
\begin{equation} \label{mapofcommrings} \pi_0(A)/(x, x_1, \dots, x_n) \to
\pi_0(A//x)/(x_1, \dots, x_n)  \end{equation}
is a finite universal homeomorphism.
The left-hand-side of \eqref{mapofcommrings} is the residue field
$k$
of $\pi_0(A)$ at the maximal ideal, and the right-hand-side is a finite module
over the left-hand-side and is in particular a product of local artinian
$k$-algebras. 
By replacing  $A$ with $A \otimes_k
k'$ for a finite extension $k'/k$, we may assume  that each of the residue fields of
the left-hand-side of
\eqref{mapofcommrings} is $k$ itself. Thus, it suffices to show
$\pi_0(A//x)/(x_1, \dots, x_n)$ has no nontrivial idempotents in this case.
As a result, our claim follows from
\Cref{purelyinsepart},
which implies that 
the connected components of 
$\spec \pi_0(A//x)/(x_1, \dots, x_n)$ are in bijection with those of 
$\spec \pi_0(A//(x, x_1, \dots, x_n))$, and in turn with those of $\spec
\pi_0(A)/(x, x_1, \dots, x_n)$, while the latter is just a point. 
\end{proof} 

By induction (and transitivity), one obtains an analogous result for any
finite sequence of elements in $\pi_0 A$. 
Moreover, by replacing $A$ with $A[t_2^{\pm 1}]$, we can thus obtain a result for quotients
by even degree elements. 
We find: 

\begin{theorem} 
\label{radicial2}
Let $A$ be a rational noetherian $\e{\infty}$-ring and let $x_1, \dots, x_n \in
\pi_{\mathrm{even}}(A)$ be a sequence of elements. Then the map 
\[ \pi_{\mathrm{even}}(A)/(x_1 , \dots, x_n) \to \pi_{\mathrm{even}}(A//(x_1,
\dots, x_n)),  \]
is a finite universal homeomorphism. 
\end{theorem}

\section{Degree $-1$ elements}

\label{degminusonesec}
We will also encounter odd degree elements in homotopy, and thus, in this
section, we  consider
the free $\e{\infty}$-$\mathbb{Q}$-algebra $\sym^* {\mathbb{Q}}[-1]$ on a generator in degree $-1$. 
It is the purpose of this section to use the coincidence
(\Cref{deg-1coincidence}) that $\sym^* \mathbb{Q}[-1]
\simeq C^*(S^1; \mathbb{Q})$ to prove certain basic facts (in
\Cref{sec:comparewithloc}) about $\sym^* \mathbb{Q}[-1]
\simeq C^*(S^1; \mathbb{Q})$-modules, and ultimately about the construction $A//y$
where $A$ is a rational $\e{\infty}$-ring and $y \in \pi_{-1}(A)$.

\subsection{The free $\e{\infty}$-ring on $k[-1]$}
Let $k$ be a field of characteristic zero, which we will work over. 
Recall that the free $\e{\infty}$-$k$-algebra on $k[-1]$ is 
\[ \sym^*k[-1] \simeq \bigoplus_{n \geq 0} \left(k[-1]\right)^{\otimes n}_{h \Sigma_n}.  \]
Here $k[-1]^{\otimes n} \simeq k[-n]$, and the
$\Sigma_n$-action is via the sign representation. 
For $n \geq 2$, this action is nontrivial, and it follows that the homotopy
coinvariants are \emph{zero}. In particular, 
we find: 
\begin{corollary} 
For $\mathrm{char} \ k = 0$, the homotopy groups of $\sym^* k[-1]$
are given by:
\[ \pi_i ( \sym^*  k[-1]) \simeq \begin{cases} 
k&  \text{if }i = 0\\
k & \text{if } i =  -1 \\
0 & \text{otherwise}
 \end{cases} ,\]
 and the multiplication is determined (``square zero'' in degree $-1$). 
\end{corollary} 
There are two other $\e{\infty}$-rings which have a similar multiplication
law on their homotopy groups: 
\begin{enumerate}
\item The cochain $\e{\infty}$-ring on $S^1$, $C^*(S^1; k)$.  
\item The square-zero $\e{\infty}$-ring $k \oplus k[-1]$. 
\end{enumerate}

\begin{proposition}
\label{deg-1coincidence}
Let $k$ be a field of characteristic zero. Then there are equivalences of
$\e{\infty}$-rings $\sym^* k[-1] \simeq C^*(S^1;k) \simeq k \oplus k[-1]$.
\end{proposition}
\begin{proof} In fact, we can produce maps
\[ \sym^* k[-1] \to C^*(S^1; k), \quad \sym^*
k[-1] \to k\oplus k[-1],  \]
such that they are isomorphisms on $\pi_{-1}$ (using the universal property of
$\sym^*$), and therefore are equivalences of $\e{\infty}$-rings by inspection
of  $\pi_*$.
So, all three are equivalent. 
\end{proof}

\begin{remark} 
If one worked over $\mathbb{F}_p$, the symmetric algebra $\sym^*(
\mathbb{F}_p[-1])$ is definitely much too large to be either $C^*(S^1;
\mathbb{F}_p)$ or $\mathbb{F}_p \oplus \mathbb{F}_p[-1]$, but 
$C^*(S^1; \mathbb{F}_p)$ and $\mathbb{F}_p \oplus \mathbb{F}_p[-1]$ have the
same square-zero multiplication on homotopy groups. They are \emph{not}
equivalent as $\e{\infty}$-rings under $\mathbb{F}_p$ because the zeroth
reduced power $\mathcal{P}^0$ acts as the identity on $\pi_{-1}$ of the former
and zero on the latter. \end{remark}

More generally, if $n$ is any \emph{odd} integer, we can repeat the above
reasoning: 

\begin{corollary} 
\label{oddcoinc}
If $n$ is odd, then we have equivalences of $\e{\infty}$-rings
\begin{equation}\label{oddfree}  \sym^* k[-n]  \simeq k\oplus
k[-n],  \end{equation}
and, if $n > 0$, then these are equivalent to cochains on the $n$-sphere, $C^*(S^n; k)$. 
\end{corollary}

\subsection{Descent properties and $A//x$}

If $A$ is a rational $\e{\infty}$-ring and $x \in \pi_*(A)$ is an element in
an \emph{even} degree, then we saw in \Cref{modxmodxgr} that the construction $A//x$ was
reasonably hands-on: it gave us the underlying $A$-module $A/x$. If $x$ is
in 
odd degree, however, $A//x$ may be much bigger than $A$.
\begin{example} 
Let $x = 0 \in \pi_{-1}(A)$. Then $A//x \simeq A[t]$. 
\end{example} 
In fact, for $x $ in an odd degree, it is not  even a priori evident that if $A$ is nonzero, then
$A//x$ is also nonzero, even as $x^2 = 0$. We will prove this (and
more) using descent theory. 
We recall first a definition. 

\begin{definition} 
\label{admitsdescent}
Let $\phi\cl A \to A'$ be a morphism of $\e{\infty}$-rings. We say that 
$\phi$ \emph{admits descent} if the thick tensor-ideal that $A'$ generates, in
$\mathrm{Mod}(A)$, is all of $\mathrm{Mod}(A)$.
\end{definition} 

We
refer to \cite[\S 3-4]{galgp} for preliminaries on the notion of ``admitting
descent.'' Here is a simple example.
\begin{proposition} \label{quotdesc}
Let $A$ be a rational $\e{\infty}$-ring and let $x \in \pi_0 A$ be nilpotent.
Then the $\e{\infty}$-$A$-algebra $A//x$ admits descent over $A$. 
\end{proposition} 
\begin{proof} 
In fact, thanks to the octahedral axiom, the thick subcategory of $\md(A)$
generated by the $A$-module $A/x$ (which is equivalent to the underlying
$A$-module of $A//x$) contains $A/x^2, A/x^3, \dots, $, and eventually $A/x^N$
where $N$ is so large that $x^N = 0$. But $A/x^N \simeq A \oplus \Sigma A$ for
such $N$, and therefore the thick subcategory generated by $A//x$ actually
contains $A$. 
\end{proof}

Let $A$ be an $\e{\infty}$-ring and let $y \in \pi_{n}(A)$, with $n$ odd. Then
$y^2 = 0$, so that one would hope that the analog 
of \Cref{quotdesc} would be automatic. That is, one would hope that $A//y
\simeq A \otimes_{\sym^*(\mathbb{Q}[n])} \mathbb{Q}$ admits descent over $A$.
A priori, 
it is harder to control this tensor product, because $\mathbb{Q}$ is no longer
a perfect $\sym^* ( \mathbb{Q}[n])$-module for $n$ odd, so one cannot
imitate the above argument. However, we can still
prove the statement.  
\begin{proposition} 
\label{quotdegminusone}
If $A$ is nonzero and $y \in \pi_n A$ (for $n$ odd), then $A//y \in \clg_{A/}$
admits descent. In particular, $A//y \neq 0$.
\end{proposition} 
\begin{proof} 
The map $\sym^* \mathbb{Q}[n] \to \mathbb{Q}$ admits descent, in view of the equivalence $\sym^* ( \mathbb{Q}[n]) \simeq
\mathbb{Q} \oplus \mathbb{Q}[n]$ of \Cref{oddcoinc}. 
It follows that the map $A \to A//y$ admits descent as well by base-change.
\end{proof} 
\begin{example}
\Cref{quotdegminusone}
definitely fails for $\e{\infty}$-rings under $\mathbb{F}_p$. 
For example, if $p = 2$, then odd degree elements can be invertible (take the
Tate spectrum $\mathbb{F}_2^{t \mathbb{Z}/2}$). If $p > 2$, odd degree elements 
square to zero but can still be ``resilient.'' In the Tate spectrum $\mathbb{F}_p^{t
\mathbb{Z}/p}$, we have
\[ \pi_* ( \mathbb{F}_p^{t \mathbb{Z}/p}) \simeq \mathbb{F}_p[t_2^{\pm 1}]
\otimes E(\alpha_{-1}), \]
where the exterior generator $\alpha_{-1}$ has the property that $\beta
\mathcal{P}^0 \alpha_{-1} = t_2^{-1}$ is invertible. Thus, even though $\alpha_{-1}$
squares to zero, a basic power operation goes from it to an invertible element.
It follows that in any $\e{\infty}$-ring under $\mathbb{F}_p^{t \mathbb{Z}/p}$,
if $\alpha$ maps to zero, the whole $\e{\infty}$-ring has to be zero. Such
phenomena can never happen in characteristic zero. 
\end{example}

\subsection{Comparison with local systems}
\label{sec:comparewithloc}
In this subsection, we 
give the most important (for this paper) application of 
\Cref{deg-1coincidence}. 
We will be able to  describe the $\infty$-category of
\emph{modules} over the free algebra $\sym^* k[-1] \simeq C^*(S^1; k)$ for
$\mathrm{char} k = 0$. 
In fact, let $k$ be any field, not necessarily of characteristic zero. We will describe modules over the cochain algebra $C^*(S^1; k)$. 
For example, we will be able to give a complete classification of all perfect
modules.

We first recall a basic construction from \cite[\S 7.2]{galgp}.

\begin{definition}
Let $X$ be a space and let $\mathcal{C}$ be an $\infty$-category. We define
$\mathrm{Loc}_X(\mathcal{C}) = \mathrm{Fun}(X, \mathcal{C})$ and refer to it
as the \emph{$\infty$-category of $\mathcal{C}$-valued local systems on $X$.}
\end{definition}

\begin{cons}
\label{imbedmodlocsys}
Let $A$ be an $\e{\infty}$-ring and let $\mathcal{C} = \mathrm{Mod}(A)$. 
Let $X$ be a finite complex. 
We have a natural fully faithful, symmetric monoidal imbedding
\[ \md( C^*(X; A)) \subset \loc_{X}( \md(A)), \]
from modules over the cochain $\e{\infty}$-ring $C^*(X; A)$ into 
local systems of $A$-modules on $X$, whose image in $\loc_{X}( \md(A))$
is the localizing subcategory generated by the unit. 
\end{cons}

Informally, the functor 
sends a $C^*(X; A)$-module $M$ to the $A$-module $M
\otimes_{C^*(X;A )} A$, which lives as a local system over
$X$ (since there is an $X$'s worth of evaluation maps $C^*(X; A)
\to A$). 

In general, it is somewhat subtle to test whether an object in
$\mathrm{Loc}_X(\md(A))$ belongs to the essential image of $\md( C^*(X; A))$. 
However, if $X = S^1$ then things simplify considerably. 
 Given a local system
$N \in \loc_{S^1}(\md(A))$, the evaluation $N_q$ (for a fixed
basepoint $q\in S^1$) is an $A$-module, and $N_q$ acquires a
monodromy
automorphism $\phi$
\[ \phi\colon N_q \simeq N_q,  \]
coming from a choice of generator of $\pi_1(S^1; q)$. 

\begin{proposition}[{\cite[Remark 7.9]{galgp}}] 
\label{unipotentcrit}
Given $N$ as above, then $N$ belongs to the image of $\md( C^*(S^1; A))$ if
and only if the action of $\phi - 1$ on the homotopy groups $\pi_*(N_q)$ is
locally nilpotent. 
\end{proposition} 

It will be important for us to have the correspondence in \Cref{unipotentcrit}
in as clear terms as possible. Thus, we state the following construction. 
\begin{cons} \label{rightadjointloc}
The right adjoint $\loc_{S^1}( \md(A)) \to \md( C^*(S^1; A ))$ is given, for 
a local system $N$, by taking its global sections $\varprojlim_{S^1}  N$.
Explicitly, if $q \in S^1$ and $\phi\colon N_q \to N_q$ is as above, we have
\begin{equation} \label{lims1} \varprojlim_{S^1} N  = \mathrm{fib}\left( N_q \stackrel{\phi - 1}{\to} N_q
\right),  \end{equation}
and in particular, we can determine the homotopy groups of
$\varprojlim_{S^1} N$ via a long exact
sequence. 
\end{cons}

\begin{remark} 
This discussion is special to the case of $S^1$. For any finite complex $X$ and
any $\e{\infty}$-ring $A$, we have an
inclusion $\md( C^*(X; A)) \subset \loc_X( \md(A))$, and the
image always is contained in the subcategory of local systems satisfying an
ind-unipotence property on homotopy groups, but the precise identification of the image
 relies on
the 1-dimensionality of the circle. 
\end{remark}

Let $\mathcal{C}$ be an arbitrary $\infty$-category.
To give a local system on $S^1$ in some $\infty$-category $\mathcal{C}$ is
\emph{equivalent} to giving 
an object of that $\infty$-category and an automorphism, via $S^1 \simeq K(
\mathbb{Z}, 1)$. 
Fix 
two $\mathcal{C}$-valued local systems on $S^1$, $(x, \phi_x), (y, \phi_y)$,
where $x, y \in \mathcal{C} $ and $\phi_x\colon x \simeq x, \phi_y\colon y \simeq y$ 
are automorphisms in $\mathcal{C}$. 
Given a map $f\colon x \to y$ such that the diagram
\begin{equation} \label{commdiag}\xymatrix{
x \ar[d]^f \ar[r]^{\phi_x} &  x \ar[d]^f \\
y \ar[r]^{\phi_y} & y
},\end{equation}
commutes up to homotopy,
then we can produce a map of 
local systems $(x, \phi_x) \to (y, \phi_y)$ extending the map $f\cl x \to y$. 

\begin{remark}Specifying such a map amounts in
addition to
\emph{choosing} a homotopy to make the diagram commute, and there may be many
homotopy classes of 
such.
\end{remark}
We can state this formally: 

\begin{proposition} 
\label{isoclassess1}
Every object in $\mathrm{Loc}_{S^1}(\mathcal{C})$ is represented by a pair
$(x, \phi_x)$ where $x \in \mathcal{C}$ and $\phi_x\cl x \to x$ is an automorphism,
and two pairs $(x, \phi_x), (y, \phi_y)$ are isomorphic 
if and only if there exists an isomorphism $f\cl x \to y$ such that the diagram
\eqref{commdiag} is homotopy commutative.
\end{proposition}

We now specialize to the case where $A = k$ is a field. We will use
Construction~\ref{imbedmodlocsys}, \Cref{unipotentcrit}, and
\Cref{isoclassess1} 
to classify $C^*(S^1;k )$-modules. 
\begin{proposition} \label{shifteddisc}
Any local system $\mathcal{L} \in \loc_{S^1}( \md(k))$ 
decomposes uniquely as a direct sum $\mathcal{L} \simeq \bigoplus_{n \in \mathbb{Z}}
\mathcal{L}_n[n]$ where the fiber of $\mathcal{L}_n$ at a point of $S^1$ is
discrete.
\end{proposition} 

\begin{proof} 
Let $\mathcal{L} = (M, \phi)$.
 The $k$-module $M$ decomposes as a sum of its homotopy groups,
i.e., $M \simeq \bigoplus_{n \in \mathbb{Z}} (\pi_n M)[n]$, and the automorphism
$\phi\colon M \to M$ is determined by its behavior on its homotopy groups. 
It follows that, for each $n$, we can produce squares
\[ \xymatrix{
(\pi_n M)[n] \ar[d] \ar[r]^{\phi_*} &  (\pi_n M)[n] \ar[d] \\
M \ar[r]^{\phi} & M
},\]
which commute up to homotopy. 
Putting this together, we can produce a square
\[ \xymatrix{
\bigoplus_{n \in \mathbb{Z}}(\pi_n M)[n] \ar[d] \ar[r]^{\phi_*} &
\bigoplus_{n \in \mathbb{Z}}(\pi_n M)[n] \ar[d] \\
M \ar[r]^{\phi} & M
},\]
which commutes up to homotopy, and where the vertical maps are now equivalences. 
It follows from \Cref{isoclassess1} that the \emph{pair} $(M, \phi)$  is equivalent  to 
the
direct sum of the pairs $\{((\pi_n M)[n], \phi_*)\}_{n \in \mathbb{Z}}$, where each of these is
concentrated in a single degree. 
\end{proof}

Using the imbedding $\md( C^*(S^1; k)) \subset \loc_{S^1}(
\md(k))$, we find from this: 
\begin{corollary} 
Any $C^*(S^1; k)$-module $M$ 
admits a unique decomposition  $M \simeq \bigoplus_{n \geq 0} M_n[n]$, where $M_n \in
\md( C^*(S^1;k ))$ has the property that $M \otimes_{C^*(S^1; k)} k$ is a
discrete $k$-module. 
\end{corollary} 

In order to give a discrete local system on $S^1$, it suffices simply to give a
$k$-vector space with an automorphism. 
In the case we are interested, i.e., local systems coming from $C^*(S^1;
k)$-modules, 
it follows that to give an equivalence class of $C^*(S^1; k)$-modules
$M$
equates to giving, for each $n \in \mathbb{Z}$, a discrete $k[x]$-module on
which $x$ acts \emph{locally nilpotently}. 
The $n$th such object corresponds to $\pi_n( M  \otimes_{C^*(S^1; k)}
k)$ and $x$ is the monodromy automorphism minus the identity. 

In general, the classification of torsion modules over a PID is nontrivial, but
in the finitely generated case, we have a simple complete classification.
This leads to:

\begin{cons}
Fix $i \in \mathbb{Z}_{>0}$.
We consider the $n$-dimensional $k$-vector space $V_i = k[x]/x^n$ and the nilpotent
endomorphism given by multiplication by $x$. Let $\mathcal{V}_i$ be the
associated local system on $S^1$ with fiber $V_i$ and monodromy automorphism
$1 + x$. 
Then, by \Cref{unipotentcrit}, $\mathcal{V}_i$ corresponds to a $C^*(S^1;k
)$-module that we will denote by $N_i$.
Thus, $N_i \otimes_{C^*(S^1; k)}  k$ is discrete,
and 
$i$-dimensional, and the monodromy automorphism is unipotent with a single
Jordan block.

In order to determine the homotopy groups of $N_i$, we have to form the
associated local system, and take global sections over $S^1$, as in
Construction~\ref{rightadjointloc}. 
We have
\[ 
\pi_j( N_i) = \begin{cases} 
k&  \text{if }j = 0 \\
k & \text{if } j = -1 \\
0 & \text{otherwise}
 \end{cases} 
 .\]
 For example, $N_1 = C^*(S^1; k)$. 
\end{cons}

\begin{proposition} 
 $N_i$ is a perfect $C^*(S^1; k)$-module.
\end{proposition} 
\begin{proof} 
By construction, $N_i \otimes_{C^*(S^1;k )} k$ is a perfect $k$-module.
Moreover, $C^*(S^1; k ) \to k$ admits descent \cite[Prop. 3.35]{galgp}. Therefore, $N_i$ is a perfect
$C^*(S^1; k)$-module in view of \cite[Prop. 3.27]{galgp}. 
\end{proof}

\begin{proposition} 
Any perfect $C^*(S^1;k )$-module $M$ decomposes uniquely as a sum of copies
of shifts of the $N_i$.  
\end{proposition} 
\begin{proof} 
Given a perfect $C^*(S^1; k)$-module, the associated local system
(which is a local system of \emph{perfect} $k$-modules) splits as a
direct sum of shifts of local systems of discrete $k$-modules, by
\Cref{shifteddisc}. Each of these is determined by a finite-dimensional
$k$-vector space with an unipotent automorphism, and these are
classified as a direct sum of indecomposable ones determined by their Jordan
type, corresponding to the decomposition of a finitely generated
$k[x]$-module as a direct sum of cyclic ones. 
These summands correspond to the $C^*(S^1;k )$-modules $N_i$.
\end{proof} 

\subsection{Evenly graded $C^*(S^1; k)$-modules}
Let $k$ be a field.
We will have to work with certain non-perfect $C^*(S^1; k)$-modules
in the sequel, and here we will prove a basic technical result (\Cref{keylem})
concerning them. 
We need the following classical algebraic fact
with $R = k[x]_{(x)}$. 
\begin{theorem} \label{tordiv}
Let $R$ be a discrete valuation ring with quotient field $K$. Then any torsion divisible $R$-module
is a direct sum of copies of $K/R$. 
\end{theorem}

We refer to \cite[Theorem 6.3]{DVRbook} for a proof of \Cref{tordiv}. 
We will translate it into our setting and prove the following.
\begin{proposition} 
\label{keylem}
Let $M$ be a $C^*(S^1; k)$-module. Suppose that $\pi_i M = 0$ if $i$
is odd. Then $M$ is a direct sum of copies of even shifts of the $C^*(S^1;
k)$-module $k$ (under the map of $\e{\infty}$-rings $C^*(S^1; k) \to k$ given
by evaluation at a point). 
\end{proposition} 
\begin{proof} 
We know, first, that $M$ is a direct sum of copies of $C^*(S^1;
k)$-modules whose base-change to $k$ is shifted discrete (\Cref{shifteddisc}), so
we may assume this to begin with. 
That is, we may assume that $M \otimes_{C^*(S^1; k)} k$ is concentrated in one
degree, say $n$, in homotopy.
In particular, under the correspondence between $C^*(S^1;k )$-modules and
local systems with the unipotence property, $M$ comes from a discrete
$k[x]$-module $P_0$, on
which $x$ is locally nilpotent. Moreover, at most one of $\ker x, \coker x$ can nonzero
because of the hypothesis on the homotopy groups on $M$, and
Construction~\ref{rightadjointloc}, which describes how to get from $P_0$
to $M$. 

Since $x$ is locally nilpotent, $\ker x$ is always nonzero (if $P_0 \neq 0$),
so the conclusion must be that $\coker x = 0$, and we have an even shift of a
discrete local system. In other words, $P_0$ is a $x$-torsion divisible
$k[x]$-module. Any such is a direct sum of copies of
$k[x^{\pm 1}]/k[x]$ by \Cref{tordiv}.

It follows that if $\widetilde{M}$ is the $C^*(S^1; k)$-module corresponding
to the $k[x]$-module $k[x^{\pm 1}]/k$, then $M$ is a direct sum of even shifts
of copies of $\widetilde{M}$. It remains to argue that $k \simeq
\widetilde{M}$. In fact, our reasoning shows that $k$ must be a direct sum of
copies of $\widetilde{M}$, but clearly $k$ is indecomposable, so $k \simeq
\widetilde{M}$. 

\end{proof}

\begin{corollary} 
Let $M$ be a $C^*(S^1; k)$-module such that $\pi_i M = 0$ if $i$ is
odd. Then the $k$-module $M \otimes_{C^*(S^1; k)}k$
has the same property. 
\end{corollary} 
\begin{proof} 
This follows from \Cref{keylem}, but it could also have been seen
directly. 
\end{proof} 
\section{Residue fields}

In this section, we will prove Theorems~\ref{existencerf} and \ref{ournilpthm} on
the existence of residue fields and the detection of nilpotence. 
It will be convenient to work throughout with an extra assumption
of a degree two unit. In this case, the attachment of even cells can always
be replaced with the attachment of degree zero cells. 

\subsection{Definitions}
Let $A$ be a rational, noetherian $\e{\infty}$-ring  such that 
$\pi_2(A)$ contains a unit. 
Fix a prime ideal $\mathfrak{p} \subset \pi_0(A)$.
\begin{definition} 
A \emph{residue field} for $A$ is an object $\kappa(\mathfrak{p}) \in
\clg_{A/}$  such that:
\begin{enumerate}
\item The map $\pi_0(A) \to \pi_0( \kappa(\mathfrak{p} ))$ exhibits $\pi_0 (
\kappa( \mathfrak{p}))$ as the residue field of $\pi_0(A)$ at $\mathfrak{p}$.
\item $\pi_1( \kappa( \mathfrak{p})) = 0$.
\end{enumerate}
In particular, if $k(\mathfrak{p})$ is the residue field of $\pi_0 A $ at
$\mathfrak{p}$, then $\pi_* ( \kappa(\mathfrak{p}))$ is a Laurent series ring
on $k( \mathfrak{p})$ on a generator in degree two.
\end{definition} 

In this section, we will show that residue fields for such $\e{\infty}$-rings
exist uniquely, and are sufficient to detect nilpotence in $\md(A)$. 
The rest of the paper will use these residue fields to describe certain
invariants of $\md(A)$. 

\begin{remark}
The name ``residue field'' is appropriate because of the perfect
K\"unneth isomorphism
\[ \kappa(\mathfrak{p})_* (M) \otimes_{\kappa(\mathfrak{p})_*} \kappa(\mathfrak{p})_*( N) \simeq
\kappa(\mathfrak{p})_*(M \otimes_A N) , \quad M, N \in \md(A);  \]
indeed, there is a map from left to right which is an isomorphism for $M = N =
A$, and both sides define two-variable homology theories on $\md(A)$, so the
natural map must be an isomorphism in general. 
Alternatively, any $\kappa(\mathfrak{p})$-module is a sum of shifts of free ones. 
\end{remark}

We describe the connection with the use of residue fields as in \cite{BR2}.
Given an even periodic $\e{\infty}$-ring $A$ (not necessarily over $\mathbb{Q}$) with 
$\pi_0(A)$ \emph{regular} noetherian, it is possible to form ``residue fields''
of $A$ as $\e{1}$-algebras in $\md(R)$, by successively quotienting by a
regular sequence. These residue fields have analogous
properties of detecting nilpotence \cite[Cor. 2.6]{thick_am} and are quite useful for
describing invariants of $\md(A)$ (e.g., \cite{BR, BR2, galgp}).  
These residue fields are usually \emph{not} $\e{\infty}$-algebras in $\md(A)$.
For example, in the ``chromatic'' setting, the associated residue fields (such
as the Morava $K$-theories $K(n)$ for the $\e{\infty}$-ring $E_n$) are
almost \emph{never} $\e{\infty}$. 

Over the rational numbers, we are able to produce residue fields without such
regularity hypotheses, and as $\e{\infty}$-algebras. 
However, we will have to work a bit harder: the 
residue fields of such an $A$ will no longer in general be perfect as
$A$-modules (or as $\e{\infty}$-$A$-algebras), and we will have to use a
countable limiting procedure, together with the techniques from the previous
sections. 

Let $A$ be as above. In order to construct a residue field for $A$ for the
prime ideal $\mathfrak{p} \in \spec \pi_0 A$, we may first localize at
$\mathfrak{p}$, and assume that $\pi_0 A$ is \emph{local} and that
$\mathfrak{p}$ is the \emph{maximal ideal.}
Then, given generators $x_1, \dots, x_n \in \pi_0 A$ for $\mathfrak{p}$, we
will need to set them equal to zero by attaching 1-cells. That of course will
introduce new elements (in both $\pi_0, \pi_1$) and we will have to kill them
in turn. 
This process will be greatly facilitated by the analysis in
\Cref{degminusonesec}.

\subsection{Detection of nilpotence}
Given an $\e{\infty}$-ring $A$, we start by reviewing what it means for a
collection of $A$-algebras to \emph{detect nilpotence,} following ideas of \cite{DHS, HS}. 

\begin{definition} 
\label{defresfield}
Let $A$ be an $\e{\infty}$-ring, and let $A'$ be an \emph{$A$-ring spectrum:}
that is, an associative algebra object in the \emph{homotopy category} of
$\md(A)$. 
We say that $A \to A'$ \emph{detects nilpotence} if, whenever $T$ is an
$A$-ring spectrum, then the map
of associative rings
\[ \pi_*(T) \to \pi_*(A' \otimes_A T)  \]
has the property that any $u \in \pi_*(T)$ which maps to a nilpotent
element is
nilpotent. 
More generally, a collection of $A$-ring
spectra $\left\{A'_\alpha\right\}_{\alpha \in S}$ is said to \emph{detect nilpotence}
if any $u \in \pi_*(T)$ which maps to nilpotent elements under each
map $\pi_*(T) \to
\pi_*(A'_\alpha \otimes_A T)$ is itself nilpotent. 
\end{definition} 

For example, the {nilpotence theorem} (\Cref{fieldsdetectnilp}) states
that the Morava $K$-theories and homology (rational and mod $p$) detect
nilpotence for $A = S^0$. The original form (in \cite{DHS}) states that the $\e{\infty}$-ring
$MU$ of complex bordism detects nilpotence by itself, again over $S^0$. 

As in \cite[\S 1]{DHS}, one has the following consequences of detecting nilpotence: 

\begin{proposition} 
Let $\left\{A'_\alpha\right\}_{\alpha \in S}$ be a collection of $A$-ring
spectra that detect nilpotence. 
\begin{enumerate}
\item Given a map of perfect $A$-modules $\phi\colon T \to T'$ such that
each $1_{A'_\alpha}
\otimes_A \phi\colon A'_\alpha \otimes_A T \to A'_\alpha \otimes_A T'$ is nullhomotopic as a map
of $A$-modules, then $\phi$ is smash nilpotent: $\phi^{\otimes N}\colon T^{\otimes N} \to
T'^{\otimes N}$ is nullhomotopic for $N \gg 0$. 
\item Given a self-map of perfect $A$-modules $v\colon \Sigma^k T \to T$, if
each $1_{A'_\alpha}
\otimes_A v\colon \Sigma^k (A'_\alpha \otimes_A T) \to A'_\alpha \otimes_A T $ is nilpotent in $\md(A)$, then $v$ itself is nilpotent. 
\end{enumerate}
\end{proposition} 

\begin{example} \label{detectperfect}
Suppose $A'$ is an $A$-ring spectrum that detects nilpotence. Then $A'$ cannot
annihilate any nonzero perfect $A$-module $M$; in fact, that would force the identity
$M \to M$ to be nilpotent. 
\end{example} 

Moreover, one sees:
\begin{proposition} 
\label{easynilp}
Let $A$ be an $\e{\infty}$-ring. 
\begin{enumerate}
\item Let $A_1 \to A_2 \to A_3 \to \dots $ be a
diagram of $A$-ring spectra, such that the colimit $A_\infty = \varinjlim A_i$
has a compatible structure of an $A$-ring spectrum. If each $A_i$ detects
nilpotence over $A$, then $A_\infty$ detects nilpotence over $A$. 
\item Let $\left\{A'_\alpha\right\}_{\alpha \in S}$ be a collection of
$\e{\infty}$-$A$-algebras that detect nilpotence over $A$. For each $\alpha \in
S$, let $\{A''_{\alpha \beta}\}_{\beta \in T_\alpha}$ be a collection of
$\e{\infty}$-$A'_\alpha$-algebras that detect nilpotence over $A'_\alpha$. Then
the collection $\left\{A''_{\alpha \beta}\right\}_{\alpha \in S, \beta \in
T_\alpha}$ of $A$-algebras detects nilpotence over $A$. 
\item Suppose the $\left\{A'_\alpha\right\}_{\alpha \in S}$ are a
collection of $\e{\infty}$-$A$-algebras detecting nilpotence
over $A$ such that each $\pi_*(A'_\alpha)$ is a graded field. Then an $A$-ring spectrum $A''$ detects nilpotence over $R$ if and
only if $A'' \otimes_A A'_\alpha \neq 0$ for each $\alpha \in S$. 
\end{enumerate}
\end{proposition} 
\begin{proof} 
By a \emph{graded field}, we mean a graded ring which is either a field
(concentrated in degree zero) or $k[t^{\pm 1}]$ for $|t|>0$ and $k$ a field.
The third assertion then follows from the second, since any nonzero ring
spectrum over each $A'_\alpha$ detects nilpotence over $A'_\alpha$. 
The proofs of the first and second assertions are straightforward. 
\end{proof} 
Finally, we need an important example of a pair that detects nilpotence. 
\begin{example} 
\label{quotloc}
Let $A$ be a rational $\e{\infty}$-ring, and let $x \in \pi_0(A)$. As 
before, the cofiber $A/x$ inherits the canonical structure of an
$\e{\infty}$-algebra under $A$, as $A//x$. The
localization $A[x^{-1}]$ always inherits a natural $\e{\infty}$-ring structure.
The claim is that the pair of $A$-algebras $\left\{A/x, A[x^{-1}]\right\}$
detects nilpotence. 

To see this, let $T$ be an $A$-ring spectrum, and let $\alpha \in \pi_j(A)$.
Suppose $\alpha$ maps to zero in $A[x^{-1}] = T \otimes_A A[x^{-1}]$. This
means that $x^N \alpha = 0$ for $N$ chosen large enough. Suppose also that
$\alpha$ maps to zero in $\pi_j(T/x) = \pi_j(T \otimes_A A/x)$. This means that
$\alpha = x\beta$ for some $\beta \in \pi_j(T)$. We then have
\[ \alpha^{2N} = \alpha^N \alpha^N = (x \beta)^N \alpha^N = \beta^N x^N
\alpha^N = 0,  \]
since $x^N \alpha = 0$. In other words, $\alpha$ is nilpotent. 

This example will be extremely important to us in making induction arguments on
the Krull dimension. 
\end{example}

\begin{example} 
Let $A \to A'$ be a map of $\e{\infty}$-rings. Suppose that $A \to A'$
{admits descent} (\Cref{admitsdescent}). Then $A \to A'$
detects nilpotence; see for instance \cite[Prop. 3.26]{galgp}. 
\end{example}

\subsection{The main result}
In this subsection, we prove the main technical result of this paper
(\Cref{residueflds}): the
existence of residue fields and the detection of nilpotence. 
We begin with a preliminary technical result. 

\begin{proposition} \label{countablerational}
Let $B$ be a rational $\e{\infty}$-ring such that: 
\begin{enumerate}
\item 
$\pi_0(B)$ is a field $k$. 
\item  $\pi_2(B)$ contains a unit $u$.  
\item  $\pi_{-1}(B)$ is a countably dimensional $k$-vector space. 
\end{enumerate} 
Then there exists a sequence of $\e{\infty}$-rings
\begin{equation} \label{sequence} B = B^{(0)} \to B^{(1)} \to B^{(2)} \to
\dots   \end{equation}
such that:
\begin{enumerate}
\item Each $B^{(i)} $ satisfies the three hypotheses above on $B$.
\item There exists an element $y_i \in \pi_{-1}(B^{(i-1)})$ such that $B^{(i)}
\simeq B^{(i-1)}//y_i$.
\item Given any element $y \in \pi_{-1}(B)$, there exists $N$ such that $y$
maps to zero under the map $B \to
B^{(N)}$. 
\end{enumerate}
\end{proposition} 
\begin{proof} Note first that, by \Cref{secresfld}, $B$ naturally admits the
structure of an $\e{\infty}$-$k$-algebra. 
Let $u_1, u_2, \dots \in \pi_{-1}(B)$ 
be a $k$-basis. We define
$B^{(1)} \simeq B//u_1 =  B \otimes_{\sym^* k[-1]} k$ via the map
$\sym^* k[-1] \to B$ classifying $u_1$. Next, we define the
$\e{\infty}$-$B^{(1)}$-algebra $B^{(2)} \simeq B^{(1)}//u_2 \simeq B^{(1)} \otimes_{\sym^*
k[-1]}k $ where the map $\sym^* k[-1] \to B^{(1)}$
classifies $u_2$. 
Inductively, we obtain a sequence
of $\e{\infty}$-rings $B^{(i)}$.
We need to verify the above three conclusions on the $B^{(i)}$. The second and
third conclusions are immediate from the construction. 

Consider the
$\sym^* k[-1]$-module $B$ under
the map $\sym^* k[-1] \to B$ classifying $u_1$. We have a map
\[ k[t_2^{\pm 1}] \otimes_k \sym^* k[-1]  \to B, \]
of $\sym^* k[-1]$-modules. The cofiber $C$, by hypothesis, is a $\sym^*
{k}[-1]$-module whose homotopy groups  are concentrated \emph{entirely}
in odd degrees. 
In particular, we find by \Cref{keylem} that $C$ is a direct sum of \emph{odd} shifts of the $\sym^*
k[-1]$-module $k$, so the cofiber sequence
\[ k[t_2^{\pm 1}] \to B \otimes_{\sym^*k[-1]} k \to C \otimes_{\sym^*k[-1]} k,  \]
(which has to induce \emph{split} exact sequences on the level of homotopy
groups)
implies that, at the level of homotopy groups, $B^{(1)} = B \otimes_{\sym^* k[-1]} k$ has
the same property as did $B$: the even homotopy groups are given by the Laurent
series ring. 
Moreover, the map $\pi_*(C) \to \pi_*(C \otimes_{\sym^* k[-1] } k)$ is
injective, so the $\left\{u_2, u_3, \dots, \right\}$ remain nonzero and
linearly independent in $\pi_{-1} (B^{(1)})$. 
We find inductively 
that all the $B^{(i)}$ have homotopy
groups entirely in odd degrees except for the Laurent series over $k$. 

\end{proof} 

\begin{lemma} 
\label{laurentseriesdetbyhomotopy}
Let $k$ be a  field of characteristic zero 
and let $A$ be any $\e{\infty}$-ring. Then the map
\begin{equation} \label{mapdetbyhomotopylaurent} \hom_{\clg} ( k[t_2^{\pm 1}], A) \to \hom_{\mathrm{Ring}_*} (
\pi_*(k[t_2^{\pm 1}]),\pi_*(A))  \end{equation}
is a bijection.
\end{lemma} 
\begin{proof} 
This follows as in the proof \Cref{secresfld}. That is, using a transcendence
basis for $k$ over $\mathbb{Q}$, one sees that  there exists a free
$\e{\infty}$-ring on a discrete $\mathbb{Q}$-module $V$ such that $k$ is a
filtered colimit of \'etale $\sym^*(V)$-algebras. 
The analog of \eqref{mapdetbyhomotopylaurent} is easily seen to be true for $\sym^*(V)$ and, by
the theory of \'etale extensions \cite[\S 7.5]{higheralg}, it must hold
for $k$.
In other words, 
the map
\begin{equation} \label{mapdetbyhomotopylaurent2} \hom_{\clg} ( k, A) \to \hom_{\mathrm{Ring}_*} (
k,\pi_0(A))  \end{equation}
is a bijection. 
It is similarly easy to see that the analog holds for $k$ replaced by
$\mathbb{Q}[t_2^{\pm 1}]$. 
Since $k[t_2^{\pm 1}] \simeq k \otimes_{\mathbb{Q}} \mathbb{Q}[t_2^{\pm 1}]$,
we may conclude. 
\end{proof} 

\begin{proposition} \label{nilpcountable}
Let $B$ satisfy the hypotheses of \Cref{countablerational}. 
\begin{enumerate}
\item Then there exists a
map of $\e{\infty}$-rings $B \to k[t_2^{\pm 1}]$ which detects nilpotence.
\item If $L$ is any field of characteristic zero, 
then the map
\begin{equation} \label{mapsBlaurent} \pi_0 \hom_{\clg}(B, L[t_2^{\pm 1}]) \to \hom_{\mathrm{Ring}_*}( \pi_*(B),
\pi_* ( L[t_2^{\pm 1}]))  \end{equation}
is a bijection. 
\end{enumerate}
\end{proposition} 
\begin{proof} 
Given $B$ and the basis $u_1, u_2,\dots, $ as above, we let $B_1$ be the
colimit $\varinjlim B^{(i)}$ of the sequence \eqref{sequence} obtained by
applying \Cref{countablerational}. Then $B_1$ satisfies the hypotheses of
this proposition as well. Moreover, the map
$\pi_{-1} (B) \to \pi_{-1}( B_1)$ is the zero map. Thus, we can repeat the above
sequential construction of \Cref{countablerational} to $B_1$ to produce
a new sequence
\[ B_1 \to B_1^{(1)} \to B_1^{(2)} \to \dots,  \]
obtained by iteratively coning off the degree $-1$ elements in $\pi_*(B_1)$. Let the
colimit be the $\e{\infty}$-$B_1$-algebra $B_2$. Then $B_2$ satisfies the
hypotheses of \Cref{countablerational}, but $\pi_{-1}(B_1) \to \pi_{-1}(B_2)$ is zero. Repeating the process, we get
a sequence
\begin{equation} \label{seq2} B \to B_1 \to B_2  \to \dots, \end{equation}
whose colimit, finally, is the $\e{\infty}$-ring $ k[t_2^{\pm 1}]$, since each of the maps is zero on
$\pi_{-1}$. 

We need to see that the map $B \to k[t_2^{\pm 1}]$ thus constructed detects nilpotence. For
this, it suffices to argue that each $B_i \to B_{i+1}$ in the above sequence
detects nilpotence, since detecting nilpotence is preserved in filtered
colimits. But $B_i \to B_{i+1}$ is a filtered colimit of maps each of which is
obtained by coning off a degree $-1$ element, and these maps detect
nilpotence by \Cref{quotdegminusone}. 

Finally, we need to analyze homotopy classes of maps $B \to L[t_2^{\pm 1}]$ where $L$ is a field of
characteristic zero. We have already shown that \eqref{mapsBlaurent} is a
surjection, so we only need to prove injectivity.
For this, we observe that any
map $B \to L[t_2^{\pm 1}]$ extends over $B \to B_1$. Indeed, $B \to B_1$ is a
filtered colimit of maps $B^{(i-1)} \to B^{(i-1)}//y_i$ where $|y_i| = -1$.
Since $\pi_{-1} L[t_2^{\pm 1}] = 0$, any map $B^{(i-1)} \to L[t_2^{\pm 1}]$
can be extended over $B^{(i)}$. In particular, 
\( \pi_0 \mathrm{Hom}_{\clg}(B^{(i)}, L[t_2^{\pm 1}]) \to   \pi_0
\mathrm{Hom}_{\clg}(B^{(i-1)}, L[t_2^{\pm 1}]) \)
is a surjection. Taking inverse limits, it follows easily that
\( \pi_0 \mathrm{Hom}_{\clg}(B_1, L[t_2^{\pm 1}]) \to   \pi_0
\mathrm{Hom}_{\clg}(B, L[t_2^{\pm 1}])
\)
is a surjection too, as claimed. 
Applying this to $B_i$, we find that each map
\(
 \pi_0 \mathrm{Hom}_{\clg}(B_i, L[t_2^{\pm 1}]) \to   \pi_0
\mathrm{Hom}_{\clg}(B_{i-1}, L[t_2^{\pm 1}])
,\)
is a surjection, and taking limits, that the map
\[ \pi_0 \mathrm{Hom}_{\clg}(k[t_2^{\pm 1}], L[t_2^{\pm 1}]) \to   \pi_0
\mathrm{Hom}_{\clg}(B, L[t_2^{\pm 1}])
  \]
  is a surjection, too. However, maps $k[t_2^{\pm 1}] \to L[t_2^{\pm 1}]$
  of $\e{\infty}$-rings
  are determined by their action on homotopy groups in view of
  \Cref{laurentseriesdetbyhomotopy}. This proves uniqueness and completes the proof of 
  \Cref{nilpcountable}.  
\end{proof}

\begin{proposition} 
\label{rsfldforart}
Let $A_0$ be a noetherian, rational $\e{\infty}$-ring containing a unit  in
degree two. 
Suppose $\pi_0(A_0)$ is  a local artinian ring with residue field $k$. 
Let $L$ be any field of characteristic zero.
Then:
\begin{enumerate}
\item The natural map $$\pi_0 \hom_{\clg} ( A_0, L[t_2^{\pm 1}]) \to
\hom_{\mathrm{Ring}_*}( \pi_* A_0 , L[t_2^{\pm 1}]) \simeq
\hom_{\mathrm{Ring}_*}( k[t_2^{\pm 1}], L[t_2^{\pm 1}])$$ is a bijection. 
\item Any map $A_0 \to L[t_2^{\pm 1}]$  of $\e{\infty}$-rings detects nilpotence. 
\end{enumerate}
\end{proposition} 

\begin{proof} 
We first treat existence in case $L = k$. Our goal is to produce a map of $\e{\infty}$-rings
from $A_0$ to $k[t_2^{\pm 1}]$. 
For this, we will need to kill  the degree zero elements, 
and the degree $-1$ elements. 
We will first kill the degree $0$ elements by adding cells in dimension one, to
reduce to the case where there is nothing (except for $k$) in
odd dimensions. Then, we will use a separate argument to kill the odd homotopy. 

We first make $A_0$ into an $\e{\infty}$-$k$-algebra, using \Cref{secresfld}.
Given $A_0$, let $\mathfrak{m} \subset \pi_0(A_0)$ be the maximal ideal, which
is a finite-dimensional $k$-vector space. Consider the map
\[ \sym^* ( \mathfrak{m}) \to A_0,  \]
of $\e{\infty}$-$k$-algebras, and form the pushout
\[ A_1 \stackrel{\mathrm{def}}{=}A_0 \otimes_{\sym^*( \mathfrak{m})} k,  \]
which has the same property as $A_0$: $\pi_*(A_1)$ satisfies the desired
noetherianness assumptions (in fact, all the homotopy groups are
finite-dimensional $k$-vector spaces), and $\pi_0(A_1)$ is local
artinian with residue field $k$ by \Cref{radicial2}. Note that $A_0 \to A_1$ admits descent in view
of \Cref{quotdesc}. 

Let $\mathfrak{m}_1 \subset \pi_0(A_1)$ be the maximal ideal, and continue the
process with
\[ A_2 \stackrel{\mathrm{def}}{=}A_1 \otimes_{\sym^*(\mathfrak{m}_2)} k,  \]
and
repeating this, we find a sequence
of 2-periodic, noetherian $\e{\infty}$-rings
\[ A_0 \to A_1 \to A_2 \to \dots ,  \]
where each $A_i$ has the following properties: 
\begin{enumerate}
\item $\pi_0(A_i)$ is a local artinian ring with residue field
$k$ and $\pi_1(A_i)$ is a finite-dimensional
$k$-vector space. 
\item $A_{i+1} $ is obtained from $A_i$ by attaching a 1-cell for each element
in a $k$-basis of the maximal ideal of $\pi_0(A_i)$, and thus $A_i \to A_{i+1}$ admits descent. 
\item In particular, the map $\pi_0(A_i) \to \pi_0(A_{i+1})$ annihilates the
maximal ideal of the former. 
\end{enumerate}

If we take the colimit $A_\infty = \varinjlim_{i} A_i$, we find an
$\e{\infty}$-$A$-algebra $A_\infty$ such that $\pi_0(A_\infty) =
k$. This
process of iteratively attaching 1-cells has likely introduced elements
in $\pi_{-1}$, but we know that $\pi_{-1}(A_\infty)$ is a
\emph{countably} dimensional $k$-vector space. 
Note that $A_0 \to A_\infty$ detects nilpotence, as it is the sequential
colimit of a sequence of objects in $\clg_{A_0/}$ that admit descent. 
But now we 
can appeal to \Cref{nilpcountable} to obtain a map $A_\infty \to k[t_2^{\pm
1}]$ which detects nilpotence. 
This completes the proof of existence. 

We may handle the first assertion (of which we now only need to prove
\emph{injectivity} of the map) in a similar manner as in the
proof of \Cref{nilpcountable}. 
That is, we observe that
\[ \hom_{\clg}(A_i, L[t_2^{\pm 1}]) \to \hom_{\clg} (A_{i-1}, L[t_2^{\pm 1}]) \]
is surjective for each $i$, because $A_{i} \simeq A_{i-1}//(x_{i}^{(1)}, \dots,
x_i^{(n)})$ where the $x_i^{(j)}$ are nilpotent. In the limit, we find that 
the map
\[ \hom_{\clg}(A_\infty, L[t_2^{\pm 1}]) \to \hom_{\clg}(A_0, L[t_2^{\pm 1}])  \]
is surjective. 
But now we can apply the uniqueness statement of \Cref{nilpcountable} to
complete the proof.
That is, maps $A_\infty \to L[t_2^{\pm 1}]$ are determined by their behavior
on homotopy groups. 
So, if we had two different maps $A_0 \to L[t_2^{\pm 1}]$ inducing the same
behavior on homotopy groups, we would get two different maps $A_\infty \to
L[t_2^{\pm 1}]$ inducing the same behavior on homotopy groups, and this is a
contradiction. 
\end{proof}

\begin{proposition} 
\label{mapstogradedfields}
Let $A$ be  a noetherian, rational $\e{\infty}$-ring containing a unit in
degree two. Suppose $k$ is a field of characteristic zero. Then the map 
\[ \pi_0 \hom_{\clg}(A, k[t_2^{\pm 1}]) \to \hom_{\mathrm{Ring}_*}(
\pi_*(A) , k[t_2^{\pm 1}])  \]
is a bijection.
\end{proposition} 
\begin{proof} 
We begin with surjectivity.
Suppose given a map $\pi_0(A) \to k$. We want to realize this at the level of
$\e{\infty}$-rings.
By localizing, we may assume that $\pi_0(A)$ is a local ring and that
$\mathfrak{p} \subset \pi_0(A)$ is its maximal ideal; let $k'$ be the residue
field $\pi_0( A)/\mathfrak{p}$. By hypothesis, we have a map $k' \to
k$. Choose ideal generators $x_1,
\dots, x_n \in \mathfrak{p}$ and use them to construct a map
\[ \mathbb{Q}[u_1, \dots, u_n] \to A, \quad u_i \mapsto x_i.  \]
We let $A_0$ be the $\e{\infty}$-ring  $A \otimes_{\mathbb{Q}[u_1, \dots, u_n]}
\mathbb{Q} \simeq A//(u_1, \dots, u_n)$. 
Then we 
get a map $A \to A_0 \to k'[t_2^{\pm 1}]$ by 
\Cref{rsfldforart} realizing the map on $\pi_*$ desired. 
We can compose this with the map 
$k [t_2^{\pm 1} ] \to k[t_2^{\pm 1}]$.

Suppose now we have two distinct $\e{\infty}$-maps $A \to k[t_2^{\pm 1}]$
realizing the same map on $\pi_*$. 
The same argument as before shows that both maps extend to (necessarily
distinct) maps $A_0 \to k[t_2^{\pm 1}]$ realizing the same map on homotopy
groups, but this contradicts \Cref{rsfldforart}. 
\end{proof} 

We can now prove our main result. 

\begin{theorem} 
Let $A$ be a noetherian, rational $\e{\infty}$-ring containing a unit in
degree two. 
\label{residueflds}
\begin{enumerate}
\item 
For each prime ideal $\mathfrak{p} \subset \pi_0(A)$, a residue field
$\kappa(\mathfrak{p}) \in \clg_{A/}$ for
$A$ at $\mathfrak{p}$ exists and is unique up to homotopy. 
\item The $\e{\infty}$-$A$-algebras $\kappa(\mathfrak{p})$ detect nilpotence. 
\end{enumerate}
\end{theorem} 

\begin{proof} 
The existence and uniqueness of residue fields is a consequence of
\Cref{mapstogradedfields}, since we can of course construct them uniquely in the
setting of commutative algebra (i.e., at the level of homotopy groups).

Finally, we need to show that the residue fields detect nilpotence. 
If $\pi_0 A$ is local artinian, then we have already seen this
(\Cref{rsfldforart}). 
Now, assume the result on detection of nilpotence proved for  all noetherian
$A$ with the Krull dimension of $\pi_0(A)$ at most $n-1$. We will then prove it for
dimension $\leq n$. In fact, we may assume that $\pi_0 A$  is noetherian
\emph{local} of Krull dimension $\leq n$. Choose $x \in \pi_0 A$ such that
$\pi_0 A/(x)$ has Krull dimension $\leq n-1$. As we saw in \Cref{quotloc}, the pair
of $\e{\infty}$-$A$-algebras $A//x, A[x^{-1}]$ detect nilpotence over $A$, and
each of these is noetherian with $\pi_0$ having Krull dimension $\leq n-1$.
Therefore, the residue fields of the $\e{\infty}$-rings $A/x, A[x^{-1}]$ are sufficient to detect
nilpotence over each of them, and thus detect nilpotence over $A$ by
\Cref{easynilp}. 

Given \emph{any} rational noetherian $\e{\infty}$-ring $A$, in order to prove
that the residue fields $\left\{\kappa(\mathfrak{p})\right\}_{\mathfrak{p} \in \spec
\pi_0 A}$ detect nilpotence over $A$, it suffices to reduce to the case where
$\pi_0 A$ is \emph{local}, and thus of finite Krull dimension, so that what we
have already done suffices to show that the residue fields detect nilpotence.

\end{proof}

\begin{proposition} 
\label{whenzero}
Let $A$ be a rational noetherian $\e{\infty}$-ring containing a unit in
$\pi_2$. Suppose $B \in \clg_{A/}$ is noetherian as well. Let $\mathfrak{p}
\in \spec \pi_0 A$. 
Let $k(\mathfrak{p})$ be the associated algebraic residue field and let
$\kappa( \mathfrak{p}) \in \clg_{A/}$ be the topological one.
Then the following are equivalent: 
\begin{enumerate}
\item  $\pi_0(B) \otimes_{\pi_0(A)} k(\mathfrak{p}) \neq 0$.
\item $B \otimes_A \kappa(\mathfrak{p}) \neq 0$. 
\end{enumerate}
\end{proposition} 
\begin{proof} 
There is a map of commutative rings
\( \pi_0(B) \otimes_{\pi_0(A)} k(\mathfrak{p}) \to \pi_0 ( B \otimes_A
\kappa( \mathfrak{p})),  \)
so if the latter is nonzero, clearly the former is as well. Suppose the former
is nonzero now. Without loss of generality, we may assume that 
$\pi_0(A)$ is local and $\mathfrak{p}$ is its maximal ideal. Let $(x_1, \dots,
x_n)$ be a system of generators for $\mathfrak{p}$ and let $A_0 = A//(x_1,
\dots, x_n)$. Then $B \otimes_A A_0 \neq 0$ as $\pi_0(B)/(x_1, \dots, x_n)
\neq 0$,  in view of \Cref{radicial2}; it is
here that we use that $B$ is noetherian. 
However, the map $A_0 \to \kappa(\mathfrak{p})$ detects nilpotence as
$\pi_0(A_0)$ is local artinian, so that if $B \otimes_A A_0 \neq 0$, then $B
\otimes_A \kappa(\mathfrak{p}) \neq 0$ too. 
\end{proof}

\begin{remark} 
\Cref{whenzero} is false if we do not assume $B$ is noetherian. 
We refer to \Cref{bigcompact} for a counterexample: the map $\mathbb{Q}[x,y]
\to A$ constructed is an isomorphism on $\pi_0$, but $B//(x,y) = 0$.
\end{remark} 

\begin{corollary} 
Let $A$ be a rational noetherian $\e{\infty}$-ring containing a unit in
$\pi_2$. Given two different prime ideals $\mathfrak{p},
\mathfrak{q} \subset \pi_0 A$, the tensor product $\kappa(\mathfrak{p}) \otimes_A
\kappa(\mathfrak{q})$ is contractible.\end{corollary}

\begin{corollary} 
Let $A$ be a rational noetherian $\e{\infty}$-ring containing a unit in degree
two. Let $A', A'' \in \clg_{A/}$ be noetherian. 
Then the following are equivalent: 
\begin{enumerate}
\item $\pi_0(A') \otimes_{\pi_0(A)} \pi_0(A'') \neq 0$.
\item $A' \otimes_{A} A'' \neq 0$.
\end{enumerate}
\end{corollary} 
\begin{proof} 
Consider the $\pi_0(A)$-algebra 
$R = \pi_0(A') \otimes_{\pi_0(A)} \pi_0(A'')$. Since it is a nonzero algebra,
there exists $\mathfrak{p} \in \spec \pi_0(A)$ such that $R \otimes_{\pi_0(A)}
k(\mathfrak{p}) \neq 0$, where $k(\mathfrak{p})$ is the residue field of
$\pi_0 A$ at $\mathfrak{p}$. 
Thus, $\pi_0(A') \otimes_{\pi_0(A)} k(\mathfrak{p}) \neq 0$ and 
$\pi_0(A'') \otimes_{\pi_0(A)} k(\mathfrak{p}) \neq 0$. By \Cref{whenzero}, 
the $\e{\infty}$-rings $A' \otimes_A \kappa(\mathfrak{p})$ and 
$A'' \otimes_A \kappa(\mathfrak{p})$ are nonzero, and thus their relative tensor
product over $\kappa(\mathfrak{p})$ is nonzero. Thus, $(A' \otimes_A A'')
\otimes_A \kappa(\mathfrak{p}) \neq 0$, so that $A' \otimes_A A'' \neq 0$ as well.
\end{proof} 

Replacing $A$ by $A[t_2^{\pm 1}]$, we also get the following result, which
will be important for the next section. 
\begin{corollary} 
\label{tensorproductnonzero}
Let $A$ be a rational noetherian $\e{\infty}$-ring (not necessarily containing a unit in
degree two). Let $A', A'' \in \clg_{A/}$ be noetherian. Then the following are
equivalent: 
\begin{enumerate}
\item $\pi_{\mathrm{even}} (A') \otimes_{\pi_{\mathrm{even}(A)}}
\pi_{\mathrm{even}}(A'') \neq 0$.
\item $A' \otimes_{A} A'' \neq 0$.
\end{enumerate}
\end{corollary}

\begin{remark}
Given a rational, noetherian $\e{\infty}$-ring 
$A$ containing a unit in $\pi_2$, we constructed a family of residue fields
$\left\{\kappa(\mathfrak{p})\right\}_{\mathfrak{p} \in \spec A}$ that detect
nilpotence in the $\infty$-category $\md(A)$. One could ask if one has in fact
a \emph{Bousfield decomposition.} That is, if an $A$-module $M$ (not
necessarily perfect) has the
property that $\kappa(\mathfrak{p})_* (M) = 0$ for all $\mathfrak{p} \in \spec A$,
does that force $M $ to be contractible? In the regular and even periodic case,
this is known (e.g., \cite[Prop. 2.5]{thick_am}). 
The analog over the sphere fails: there are noncontractible spectra that smash
to zero with every residue field, for instance the Brown-Comenentz dual $I$ of
the sphere (\cite[Appendix B]{HoveyS}). We do not know what happens in $\md(A)$. 
\end{remark}

\section{A thick subcategory theorem}

In this section, we use \Cref{residueflds} to obtain the classification of thick subcategories
of perfect modules over a rational, noetherian $\e{\infty}$-ring.  

\subsection{Review of the axiomatic argument}
Let $A$ be an $\e{\infty}$-ring together with a collection 
$\left\{\kappa(\mathfrak{p})_*\right\}$ of multiplicative homology
theories on $\md(A)$, satisfying perfect K\"unneth isomorphisms,
such that together they detect nilpotence over $A$. 
In this case, it is well-known that they are sufficient to detect thick
subcategories as well. We note that these are the same as thick
tensor-ideals, since the unit generates $\md^\omega(A)$  as a thick subcategory. 
 In particular, consider $M , N \in
\md^\omega(A)$. 
We recall:

\begin{theorem}[{Hopkins-Smith-Hovey-Palmieri-Strickland 
\cite[Th. 5.2.2 and Cor. 5.2.3]{axiomatic}}]
\label{axiomaticarg}
Suppose $M, N \in \md^\omega(A)$ are such that 
 whenever $\kappa( \mathfrak{p})_*(M)
\neq 0$, then $\kappa(\mathfrak{p})_* (N) \neq 0$ too. Then the thick subcategory
that  $N$ generates in $\md^\omega(A)$ contains $M$.
\end{theorem} 

At least under noetherian hypotheses, every thick subcategory $\mathcal{C}\subset\md^\omega(A)$ is then determined by a subset of
the $\left\{\kappa(\mathfrak{p})_*\right\}$ (the $\kappa(\mathfrak{p})_*$ such that
there exists $X \in \mathcal{C}$ with $\kappa(\mathfrak{p})_*(X) \neq 0$), and the
classification of thick subcategories reduces to the determination of which
subsets arise from thick subcategories; or equivalently, which subsets of the
$\left\{\kappa(\mathfrak{p})\right\}$ arise as $\left\{\mathfrak{p}\colon
\kappa(\mathfrak{p})_* (M) \neq 0\right\}$ for some $M \in \md^\omega(A)$.

\subsection{The even periodic case}
The main subtleties of the present section will revolve around the grading. As a
result, we start with the simplest case, where the ring contains a unit in
degree two. 
This is a direct consequence of the work of the previous section and 
the axiomatic argument, \Cref{axiomaticarg}. 

\begin{cons}
\label{cons:thick1}
Let $A$ be a rational, noetherian $\e{\infty}$-ring containing a unit in
degree two. 
Let $Z \subset \spec \pi_0 A$ be a specialization-closed subset. We define a
thick subcategory $\md_Z^\omega(A) \subset \md^\omega(A)$ consisting of
modules $M$ such that $\pi_0(M) \oplus \pi_1(M)$ is set-theoretically
supported on a closed subset of $Z$.
\end{cons}
Construction~\ref{cons:thick1} clearly defines thick subcategories of $\md^\omega(A)$.
We start by noting that they can also be defined in terms of the residue
fields of $A$.

\begin{proposition} \label{supportfield}
Let $A$ be a rational, noetherian 2-periodic $\e{\infty}$-ring. 
Let $M$ be a perfect $A$-module. Then the following are equivalent, for
$\mathfrak{p} \in \spec \pi_0 A$: 
\begin{enumerate}
\item $M_{\mathfrak{p}}  \neq 0$. 
\item $\kappa(\mathfrak{p})_*(M) \neq 0$ (where $\kappa(\mathfrak{p}) \in
\clg_{A/}$ is the residue
field of \Cref{residueflds}). 
\end{enumerate}
\end{proposition} 

\begin{proof} 
Without loss of generality, we can assume that $\pi_0(A)$ is local and that 
$\mathfrak{p}$ is the maximal ideal of $\pi_0(A)$. 
Then we need to show that if $\kappa(\mathfrak{p})_*(M) = 0$, then $M$ itself is
contractible, a form of Nakayama's lemma. 
Without loss of generality, we can assume that $\pi_0(A)$ is \emph{complete}
local, because the completion is faithfully flat over $A$
\cite[Theorem 56, \S 24]{matsumura}. 

Let $x_1, \dots, x_n \in \pi_0(A)$ be generators for the maximal ideal
$\mathfrak{p}$. 
Then it suffices to show that $M/(x_1, \dots, x_n)$, which is the base-change
of $M$ to $A //(x_1, \dots, x_n)$, is
contractible, because $M$ is $(x_1, \dots, x_n)$-adically complete. 
In particular, we may replace $A$ with $A//(x_1, \dots, x_n)$ and thus assume
that $\pi_0 (A)$ is actually \emph{local artinian.} We thus reduce to this case. 

But if $\pi_0(A)$ is local artinian, we know that the map $A \to
\kappa(\mathfrak{p})$ actually
\emph{detects nilpotence}: in particular, it cannot annihilate a nonzero perfect
$A$-module (\Cref{detectperfect}). So, if $\kappa(\mathfrak{p})_*(M) = 0$, then $M$ is contractible. 
\end{proof}

\begin{proposition} \label{2periodicthick}
Let $A$ be a rational, noetherian $\e{\infty}$-ring containing a unit in
degree two. Then the thick subcategories of $\md^\omega(A)$ are in natural
bijection with the specialization-closed subsets of $\spec \pi_0 A$, via the
correspondence given in Construction~\ref{cons:thick1}.
\end{proposition} 
\begin{proof} 
By \Cref{axiomaticarg} and \Cref{supportfield}, it follows that if $M, N \in
\md^\omega(A)$ and the set-theoretic support of $\pi_0(M) \oplus \pi_1(M)$
contains that of $\pi_0(N) \oplus \pi_1(N)$, then $N$ belongs to the thick
subcategory generated by $M$. 

Next, we argue that
any closed subset of $\spec \pi_0 A$ is realized as the support of $\pi_0 M$
for some $M \in \md^\omega(A)$. If the closed subset is defined by the ideal
$(x_1, \dots, x_n) \in \pi_0 (A)$, then we can take $M = A/(x_1, \dots,
x_n)$, thanks to \Cref{radicial2}.
This implies that if $Z, Z' \subset \spec \pi_0 A$ are two distinct
specialization-closed subsets, say $Z \setminus Z' \neq \emptyset$ then there exists a module $M \in
\md^\omega(A)$ which belongs to $\md_Z^\omega(A) \setminus \md_{Z'}^\omega(A)$.
In other words, the map of Construction~\ref{cons:thick1} from subsets to
thick subcategories is injective. 

Finally, if $\mathcal{C} \subset \md^\omega(A)$ is a thick subcategory, we let
$Z$ be the specialization-closed subset of those $\mathfrak{p} \in \spec \pi_0
(A)$ such that there there exists $M \in \mathcal{C}$ with
$\kappa(\mathfrak{p})_* (M) \neq 0$. Clearly, $\mathcal{C} \subset
\md_Z^\omega(A)$. To see equality, fix an arbitrary $M \in \md_Z^\omega(A)$. 
Then there exists a set $\left\{M_\alpha\right\}_{\alpha \in S}$ of objects in $\mathcal{C}$
such that the support of $\pi_0 M \oplus \pi_1 M$ is contained in the union of 
the supports of the $\{\pi_0 M_\alpha \oplus \pi_1 M_\alpha\}$. Therefore,
since $\pi_0 (A)$ is noetherian,
there exists a finite subcollection $S' \subset S$ such that the same
conclusion holds, and \Cref{axiomaticarg} implies that $M$ belongs to the
thick subcategory generated by $\bigoplus_{\alpha \in S'} M_\alpha \in
\mathcal{C}$. 
\end{proof} 
\subsection{Graded rings}
In the rest of the section, we will explain how to adapt the argument of
\Cref{2periodicthick} to the general case, where we do not assume the
existence of a unit in degree two. 
We begin with a review of some facts about  graded rings. 
We will work with graded rings which 
are
commutative (in the \emph{ungraded} sense) such as $\pi_{\mathrm{even}}$ of an
$\e{\infty}$-ring. 

\begin{definition} 
Let $R_*$ be a commutative, graded ring. 
The topological space 
$\grspec(R_*)$  consists of the homogeneous prime ideals of $R_*$. The topology
on $\grspec(R_*)$ is defined by taking as a basis of open sets the subsets $V(a)
\stackrel{\mathrm{def}}{=} \left\{\mathfrak{p} \in \grspec(R_*): a \notin
\mathfrak{p}\right\}$ for each homogeneous element $\mathfrak{p}$.
Note that $\grspec(R_*) \subset \spec R_*$ and the inclusion map is continuous
(for the usual Zariski topology on the latter). 
\end{definition}

Given a graded ring $R_{*}$, the space $\grspec(R_*)$ is also the underlying
topological space \cite[Ch. 5]{LMB} of the algebraic stack $( \spec R_*)/\mathbb{G}_m$, where
the $\mathbb{G}_m$-action on $\spec R_*$ is given by the grading of $R_*$. 
Given a point of $( \spec R_*)/\mathbb{G}_m$, represented a map $\spec k \to (\spec
R_*)/\mathbb{G}_m$ where $k$ is a field, one obtains a map of graded rings
$R_* \to k[t^{\pm 1}]$ where $|t| = 1$. The kernel of this map is a
homogeneous prime ideal of $\grspec(R_*)$, which gives the correspondence
between points of the stack and $\grspec(R_*)$. 

\begin{example} 
Suppose $R_*$ is nonnegatively graded, i.e., $R_i = 0$ for $i < 0$, and
suppose $R_0$ is a field. Then $\grspec(R_*)$ is the union of
$\mathrm{Proj}(R_*)$ and one additional point, corresponding to the irrelevant
ideal $\bigoplus_{i>0}R_i$. 
\end{example} 

\begin{definition}
Let $R_*$ be a commutative, graded ring. A collection $\mathfrak{C} \subset
\grspec(R_*)$ of homogeneous prime
ideals of $R_*$ is
\emph{closed under specialization} if, whenever $\mathfrak{p} \in \mathfrak{C}$
and $\mathfrak{q} \supset \mathfrak{p}$ is a larger homogeneous prime ideal,
then $\mathfrak{q} \in \mathfrak{C}$ too.  
$\mathfrak{C} \subset \grspec(R_*)$ is closed under specialization if and only
if it is a union of closed subsets. 
\end{definition}

We next observe that there is a notion of ``support'' in the graded setting. 
\begin{definition} 
\label{grsupport}
Suppose $R_*$ is noetherian and $M_*$ is a finitely generated graded $R_*$-module. Let
$\mathrm{Supp}(M_*)$ denote the collection 
of all $\mathfrak{p} \in \grspec(R_*)$ such that the 
localization $(M_*)_{\mathfrak{p}}$ does not vanish. Then
$\mathrm{Supp}(M_*) \subset \grspec(R_*)$ is closed as it is the intersection of
the usual support of $M_*$ in $\spec(R_*)$ with $\grspec(R_*) \subset \spec(R_*)$.
\end{definition} 

A priori, the construction of the localization $(M_*)_{\mathfrak{p}}$ (which
is not a graded $R_*$-module) is somewhat 
unnatural. However, it is easy to see that the condition that 
$(M_*)_{\mathfrak{p}} \neq 0$ is equivalent to the condition that the
localization of $M_*$ at the multiplicative subset $S = \left\{r \in R_* \
\text{homogeneous} : r \notin \mathfrak{p}\right\}$ should not vanish, and
this latter localization is naturally a graded $R_*$-module.

We will next need the notion of a graded-local ring. 
\begin{definition} 
$R_*$ is \emph{graded-local} if it has a unique maximal homogeneous ideal.
$R_*$ is a \emph{graded field} if either:
\begin{enumerate}
\item $R_* = k$, concentrated in degree zero, where $k$ is a field.  
\item $R_* =k[u, u^{-1}]$ where $|u| > 0$ and $k$ is a field.
\end{enumerate}
\end{definition} 

It is easy to see that a graded ring is a graded field if and only if the
zero ideal is a maximal homogeneous ideal. 
In fact, this condition implies that any homogeneous element is either zero or a
unit. As a result, given any $\mathfrak{p}  \in \grspec(R_*)$, we can 
form the graded $R_*$-algebra $R_{*[\mathfrak{p}]}/\mathfrak{p}$, where
$R_{*[\mathfrak{p}]}$ is the localization of $R_*$ at the set of homogeneous
elements not in $\mathfrak{p}$. This is a graded field. 

We will next need to understand a little about the interaction between homogeneous and
inhomogeneous prime ideals. 
\begin{cons}
Let $\mathfrak{q} \in \spec (R_*)$ be a  prime ideal (not assumed
homogeneous). We define a homogeneous prime ideal $\mathfrak{p} \in \grspec(R_*)$ such that $x \in
\mathfrak{p}$ if and only if each homogeneous component of $x$  belongs to
$\mathfrak{q}$. Clearly, $\mathfrak{p} \subset \mathfrak{q}$ is the maximal
homogeneous ideal contained in $\mathfrak{q}$. 
\end{cons}

In the language of stacks, we have a quotient map $\spec R_* \to (\spec
R_*)/\mathbb{G}_m$, which induces a map on points $\spec R_* \to
\grspec(R_*)$. This map sends $\mathfrak{q} \mapsto \mathfrak{p}$ as above.

Finally, to use both graded and ungraded techniques, we need the following
construction.
\begin{cons}\label{Deg1gr}
Let $R_*$ be a commutative, graded noetherian ring. Let $R'_*$ be the graded
ring $R_*[u^{\pm 1}]$ where $|u| = 1$. Then graded $R'_*$-modules are
canonically in correspondence with \emph{ungraded} $R_*$-modules.
For example, let $\mathfrak{q} \in \spec(R_*)$ be a prime ideal, not assumed
homogeneous. We let $k(\mathfrak{q})_*$ denote the graded $R'_*$-module
corresponding to ungraded $R_*$-module which is the residue field of $R_*$ at
$\mathfrak{q}$.
\end{cons}

Given an ungraded $R_*$-module $M$, we can form a graded $R'_*$-module $M'_*$ as in
Construction~\ref{Deg1gr}
and restrict to get a graded $R_*$-module (still denoted $M'_*$), such that
$M'_n = M$ for every $n$.
The following elementary lemma will be crucial in the next subsection. 

\begin{lemma} 
\label{nonzeroalg}
Let $R_*$ be a graded, commutative noetherian ring. Let $R'_* = R_*[u^{\pm
1}]$ where $|u| = 1$. 
Let $\mathfrak{q} \in \spec R_*$ and let $\mathfrak{p} \in \grspec(R_*)$ be
the homogeneous part. 
Define the graded $R'_*$-algebras $k(\mathfrak{p})_*, k(\mathfrak{q})_*$
as 
in Construction~\ref{Deg1gr}.
Then $k(\mathfrak{p})_* \otimes_{R_*}
k(\mathfrak{q})_* \neq 0$.
\end{lemma} 
\begin{proof} 
We can assume without loss of generality that $R_*$ is graded-local with
maximal homogeneous ideal $\mathfrak{p}$. After replacing $R_*$ with
$R_*/\mathfrak{p}$, we can assume $\mathfrak{p} =0$. In this case, $R_*$ is a
graded field so that the tensor product of any two nonzero graded
$R_*$-modules (e.g., $k(\mathfrak{p})_*, k(\mathfrak{q})_*$, under pull-back
from $R_* \to R'_*$) is nonzero. 
\end{proof}

\subsection{The thick subcategory theorem}

Let $A$ be a rational, noetherian
$\e{\infty}$-ring. 
The purpose of this subsection is to give the proof that thick subcategories
of $\md^\omega(A)$ correspond to specialization-closed subsets of $\grspec(
\pi_{\mathrm{even}}(A))$, without assuming the existence of a unit in degree
two. 
We first state formally the map that realizes the correspondence.
We note that our results prove that the map from the \emph{spectrum} (cf.
\cite{Bal05}) of the
$\otimes$-triangulated category associated to $\md^\omega(A)$ to the
homogeneous spectrum of $\pi_*(A)$, as constructed by Balmer \cite{Balmer2}, is
an isomorphism. 

\begin{definition} 
Let $M \in \md^\omega(A)$. We define the \emph{support} $\mathrm{Supp}(M)$ to
be the support (in the sense of \Cref{grsupport}) of the graded $\pi_{\mathrm{even}}(A)$-module $\pi_*(M)$.
We will denote this by $\supp M$; it is a subset of $\grspec(
\pi_{\mathrm{even}}(A))$.
Given a specialization-closed subset $ Z\subset \grspec(
\pi_{\mathrm{even}}(A))$, we define $\md_Z^\omega(A) \subset \md^\omega(A)$ to
be the full subcategory spanned by those perfect $A$-modules $M$ with
$\mathrm{Supp} M \subset Z$. 

\end{definition}

\begin{theorem} 
\label{thickwithoutperiodic}
Let $A$ be a rational, noetherian $\e{\infty}$-ring.
Then the construction $Z \mapsto \md_Z^\omega(A)$ defines a correspondence between the thick subcategories 
of $\md^\omega(A)$ and specialization-closed subsets of $\grspec(
\pi_{\mathrm{even}}(A))$. 
\end{theorem}

The primary goal of this section is to give a proof of \Cref{thickwithoutperiodic},
which we already did (in \Cref{2periodicthick}) in case $\pi_2(A)$ contains a unit. 

To begin with, we will need to discuss residue fields for $A$. 
Let $A$ be a noetherian rational $\e{\infty}$-ring and $\mathfrak{p}
\subset \pi_{\mathrm{even}}(A)$ a homogeneous prime ideal. We form the
$\e{\infty}$-ring $A' = A[t_2^{\pm 1}]$  and we then
have
\[  \pi_0 A' \simeq \pi_{\mathrm{even}} A. \]
In particular, $\mathfrak{p}$ becomes a prime ideal of $\pi_0 A'$. As a result,
in view of \Cref{residueflds}, we can construct a residue field
$\kappa(\mathfrak{p})$ of $A'$ at $\mathfrak{p}$ and we obtain maps $A \to A' \to
\kappa(\mathfrak{p})$. Rather than considering the
$\left\{\kappa(\mathfrak{p})\right\}$ as $\e{\infty}$-$A'$-algebras, we consider them as
$\e{\infty}$-$A$-algebras. They satisfy a perfect K\"unneth isomorphism as
homology theories on $\md(A)$, as
before. 

\begin{lemma} \label{tensornonzeroring}
Let $\mathfrak{q} \subset \pi_{\mathrm{even}} (A)$ be an inhomogeneous prime ideal
and let $\mathfrak{p} \subset \pi_{\mathrm{even}}(A)$ be the homogeneous part.
Let $\kappa(\mathfrak{q}), \kappa(\mathfrak{p}) \in \clg_{A'}$ denote the respective
residue fields, which we regard as $\e{\infty}$-$A$-algebras under $A \to A'$.
Then $\kappa(\mathfrak{q}) \otimes_A \kappa(\mathfrak{p}) \neq 0$. 
\end{lemma} 
\begin{proof} 
This follows from \Cref{tensorproductnonzero} 
and \Cref{nonzeroalg}. 
\end{proof} 

As a result, we can now show that the residue fields $\kappa( \mathfrak{p})
\in \clg_{A/}$
for $\mathfrak{p} \in \grspec(\pi_{\mathrm{even}}(A))$ suffice to detect
nilpotence. 
\begin{theorem} \label{homog_nilp}
The $\left\{\kappa(\mathfrak{p})\right\}$, as $\mathfrak{p}$ ranges over the
homogeneous prime ideals in $\spec \pi_{\mathrm{even}} (A)$, detect
nilpotence over $A$. 
\end{theorem} 
This is not an immediate consequence of \Cref{residueflds} applied to $A'$,
because the homogeneous $\mathfrak{p}$ do not exhaust all the prime ideals of
$\pi_0(A')$. In other words, the $\kappa(\mathfrak{p})$ in question do not detect
nilpotence over $A'$. 
\begin{proof} 
We know that the $\left\{\kappa(\mathfrak{q})\right\} \subset \clg_{A'}$ for $\mathfrak{q} \in \spec
\pi_{\mathrm{even}} A = \pi_0 A'$ detect nilpotence over $A'$
(\Cref{residueflds}), and thus over
$A$. 
Thus, in order to prove that the $\left\{\kappa(\mathfrak{p})\right\}$ for
$\mathfrak{p}$ ranging over the \emph{homogeneous} prime ideals detect
nilpotence over $A$, we appeal to the third part of \Cref{easynilp} 
and \Cref{tensornonzeroring}.
\end{proof} 

We note now that if $M \in \md^\omega(A)$, then the support of $M$ in
$\grspec(\pi_{\mathrm{even}}(A))$ is equivalently the set of $\mathfrak{p} \in
\grspec(\pi_{\mathrm{even}}(A))$ such
that $M \otimes_A \kappa(\mathfrak{p}) \neq 0$; this is a consequence of
\Cref{supportfield}.

\begin{proof}[Proof of \Cref{thickwithoutperiodic}]
In particular, it now follows formally (via the axiomatic argument
 given in \Cref{axiomaticarg}, together with the noetherianness of $\grspec(
\pi_{\mathrm{even}}(A)$) from \Cref{homog_nilp} that a thick
subcategory of $\md^\omega(A)$ is determined by a subcollection of the
$\left\{\kappa(\mathfrak{p})\right\}$ as $\mathfrak{p}$ ranges over the homogeneous
prime ideals of $\pi_{\mathrm{even}}(A)$. It remains to determine what subsets
are allowed to arise. 

We will show that those subsets are precisely those which
are closed under specialization. To see this, we need to show that every
closed subset of $\grspec( \pi_{\mathrm{even}}(A))$ (associated to a
homogeneous
ideal $I \subset \pi_{\mathrm{even}}(A)$) can be realized as the support of some $M$,
but this follows by forming $A/(x_1, \dots, x_n)$ where $x_1, \dots, x_n \in
\pi_{\mathrm{even}}(A)$ generate $I$. In particular, this completes the proof
of \Cref{thickwithoutperiodic}. 
\end{proof}

\section{Galois groups}

Let $A$ be an $\e{\infty}$-ring such that $\pi_0 (A)$ has no nontrivial
idempotents. In \cite{galgp}, we introduced the \emph{Galois group} $\pi_1
\md(A)$ of $A$, a
profinite group defined ``up to conjugacy'' (canonically as a profinite
\emph{groupoid}), by developing a version of the \'etale fundamental group
formalism. The Galois group 
$\pi_1 \md(A)$ has the property that if $G$ is a finite group, then to give a
continuous group homomorphism $\pi_1 \md(A) \to G$ is equivalent to giving a 
faithful $G$-Galois extension of $A$ in the sense of Rognes \cite{rognes}. 
More generally, we introduced (\cite[Def. 6.1]{galgp}) the notion of a \emph{finite cover}
of an $\e{\infty}$-ring $A$, as a homotopy-theoretic version of the classical
notion of a finite \'etale algebra over a commutative ring. A continuous action
of $\pi_1 \md(A)$ on a finite set is equivalent to a finite cover of the
$\e{\infty}$-ring $A$. 
The Galois group can be a fairly sensitive invariant of
$\e{\infty}$-rings; for instance (\cite[Ex. 7.21]{galgp}) two different
$\e{\infty}$-structures on the same $\e{1}$-ring can yield different Galois
groups, and computing it appears to be a subtle problem in general. 
Here, we will show that the Galois group is much less sensitive over the
rational numbers, under noetherian hypotheses.

The Galois group
comes with a surjection
\begin{equation} \label{galsurj} \pi_1 \md(A) \twoheadrightarrow
\pi_1^{\mathrm{et}} \spec \pi_0 (A), \end{equation}
since every algebraic Galois cover of $\spec \pi_0 (A)$ can be realized
topologically. More generally, to every finite \'etale $\pi_0 (A)$-algebra $A'_0$
one can canonically associate $A' \in \clg_{A/}$ such that $\pi_0 A' \simeq
A'_0$ and such that $\pi_k A' \simeq A'_0 \otimes_{\pi_0 A} \pi_k A$
\cite[\S 7.5]{higheralg}. 
This yields a full subcategory of the category of finite covers which corresponds to
the above surjection. 
In general, however, it is an insight of Rognes that the above
surjection has a nontrivial kernel: that is, there exist finite covers that
do not arise algebraically in this fashion. A basic example is the
complexification map $KO \to
KU$. 

In \cite{galgp}, we computed Galois groups in certain instances. Our basic
ingredient (\cite[Th. 6.30]{galgp}) was a strengthening of work of
Baker-Richter \cite{BR2} to show that the Galois theory is entirely algebraic
for even periodic $\e{\infty}$-rings with \emph{regular} $\pi_0$, using the
theory of residue fields. Over the rational numbers, the methods of the present
paper enable one to construct these ``residue fields'' without
regularity assumptions. 
In particular, we will show in this section 
that, for noetherian rational $\e{\infty}$-rings, 
the computation of the Galois group can be reduced to a problem of pure
algebra. (For instance, we will show that \eqref{galsurj} is an isomorphism if
$A$ contains a unit in degree two.) In general, \eqref{galsurj} will not be an isomorphism, because over
$\mathbb{Q}$, it is permissible to adjoin square roots of invertible elements
in homotopy in degrees divisible by four, for instance. But we will see that
such issues of grading are the \emph{only} failure of \eqref{galsurj} to be an isomorphism. 

\subsection{Review of invariance properties}

To obtain the results of the present section, we will need some basic tools for
working with Galois groups, which will take the form of the ``invariance
results'' of \cite{galgp}. 
For example, we will need to know that killing a nilpotent degree zero class
does not affect the Galois group. 
For convenience, we will assume that all $\e{\infty}$-rings $A$ considered in
this section have no nontrivial idempotents in $\pi_0$, so that we can speak
about a Galois group.

\begin{theorem} \label{nilpinv}
Let $A$ be a rational $\e{\infty}$-ring and let $x \in \pi_0 A$ be a nilpotent
element. Then the map $A \to A//x$ induces an isomorphism on Galois groupoids. 
\end{theorem} 
\begin{proof} 
This is \cite[Theorem 8.13]{galgp}, for the map $\mathbb{Q}[[t]] \to A$ sending $t
\mapsto x$. 
\end{proof}

\begin{proposition} \label{1surj}
Let $A$ be a rational $\e{\infty}$-ring and let $x \in \pi_{-1} A$ be a class.
Then the map
\[ A \to A//x \simeq A \otimes_{\sym ^* \mathbb{Q}[-1]} \mathbb{Q},  \]
obtained by coning off $x$, induces a surjection on Galois groups. 
\end{proposition} 
\begin{proof} 
By \cite[\S 8.1]{galgp}, it suffices to show that the map $C^*(S^1; \mathbb{Q}) \simeq
\sym^* \mathbb{Q}[-1] \to \mathbb{Q}$ is \emph{universally connected}: that is,
for any  $A \in \clg_{C^*(S^1; \mathbb{Q})/}$, the natural map
\( A \to A \otimes_{C^*(S^1; \mathbb{Q})}  \mathbb{Q} \)
induces an isomorphism on $\mathrm{Idem}$. 

Since $C^*(S^1; \mathbb{Q}) \to \mathbb{Q}$ admits descent, the set $\idem(A)$
of idempotents in $A$ is the equalizer of the two maps
\[ A \otimes_{C^*(S^1; \mathbb{Q})} \mathbb{Q} \rightrightarrows 
A \otimes_{C^*(S^1; \mathbb{Q})} \mathbb{Q} \otimes_{C^*(S^1; \mathbb{Q})}
\mathbb{Q} \simeq 
\left(A \otimes_{C^*(S^1; \mathbb{Q})} \mathbb{Q} \right)[t].
\]
This is a reflexive equalizer, and one of the maps is the natural inclusion 
$$A \otimes_{C^*(S^1; \mathbb{Q})} \mathbb{Q}  \to 
\left(A \otimes_{C^*(S^1; \mathbb{Q})} \mathbb{Q} \right)[t],$$
which induces an isomorphism on idempotents. It follows that all the maps in
the reflexive equalizer are isomorphisms and thus the two forward maps are
equal, proving that 
$A \to A \otimes_{C^*(S^1; \mathbb{Q})} \mathbb{Q} $ induces an isomorphism on
idempotents. 
\end{proof}

\subsection{The periodic case}
We are now ready to show (\Cref{galep} below) that the Galois theory of a
noetherian rational $\e{\infty}$-ring containing a degree two unit is algebraic. 

\begin{lemma} 
\label{evenpdicartiniangal}
Let $A$ be a rational, noetherian $\e{\infty}$-ring  containing a unit in
$\pi_2$ such that $\pi_0 A$
is local artinian. Then the Galois theory of $A$ is algebraic. 
\end{lemma} 
\begin{proof} 
The strategy is to imitate the proof of \Cref{residueflds}, while cognizant of
the invariance results for Galois groups reviewed in the previous subsection. 
Namely, we not only showed that $A$ had a residue field, but we constructed it
via  a specific recipe. 
Let $k$ be the residue field of $\pi_0 A$. 

In proving \Cref{residueflds} (that is, in the course of the proof of
\Cref{rsfldforart}), we first formed a sequence of rational, noetherian, 
$\e{\infty}$-rings with artinian $\pi_0$,
\[ A  = A_0 \to A_1 \to A_2 \to \dots \to A_\infty = \varinjlim A_i,  \]
such that:
\begin{enumerate}
\item $A_{i+1}$ is obtained from $A_i$ by attaching 1-cells along a finite
number of nilpotent elements in $\pi_0(A_i)$.
\item All the $\pi_0(A_i)$ are local artinian rings with residue field $k$, and each map $\pi_0(A_i)
\to \pi_0(A_{i+1})$ annihilates the maximal ideal. 
\end{enumerate}

By \Cref{nilpinv}, at no finite stage do we
change the Galois group; each map $A \to A_i$ induces an isomorphism on Galois
groups. Now, by \cite[Th. 6.21]{galgp}, the Galois group is compatible with
filtered colimits and therefore $A \to A_\infty$ induces an isomorphism on
Galois groups.\footnote{In fact, all we need for the proof of this lemma is that $A \to A_\infty$ induces a
\emph{surjection} on Galois groups. This does not require the obstruction
theory used in proving \cite[Th. 6.21]{galgp}, and is purely formal.}

Now, the $\e{\infty}$-ring $A_\infty$ has the properties of
\Cref{countablerational}: it has a
unit in degree two, its $\pi_0$ is isomorphic to $k$, and $\pi_1$ is countably
dimensional. We showed in the proof of \Cref{countablerational} that by killing degree $-1$ cells repeatedly
and forming countable colimits, and repeating countably many times, we could
start with $A_\infty$ and reach $k[t_2^{\pm 1}]$. It follows by \Cref{1surj} (along
with the compatibility of Galois groups and filtered colimits, again) that the
map
\[ A_\infty \to k[t_2^{\pm 1}],  \]
induces a \emph{surjection} on Galois groups. But the Galois group of
$k[t_2^{\pm 1}]$ is algebraic (i.e., $\mathrm{Gal}(\overline{k}/k)$) in view of the K\"unneth isomorphism \cite[Prop.
6.28]{galgp}, so the Galois group of $A_\infty$ must be bounded by
$\mathrm{Gal}(\overline{k}/k)$, and therefore that of $A$ must
be, too. 
\end{proof} 

We can now prove the main result of the present subsection. 

\begin{theorem} 
\label{galep}
Let $A$ be a rational, noetherian $\e{\infty}$-ring containing a unit in
degree two. Then the Galois
theory of $A$ is algebraic, i.e., $\pi_1 \md(A) \simeq
\pi_1^{\mathrm{et}}\spec \pi_0(A)$. 
\end{theorem} 
\begin{proof} 
Fix  a finite cover $A \to A'$ of $\e{\infty}$-rings. We need to show that
$A'$ is \emph{flat} over $A$: that is, the natural map $\pi_0(A) \to \pi_0(A')$
is flat, and the map $\pi_*(A) \otimes_{\pi_0(A)} \pi_0(A') \to \pi_*(A')$ is
an isomorphism. 
This is a local question, so we may assume that $\pi_0 (A)$ is a local noetherian
ring. Moreover, by completing $A$ at the maximal ideal $\mathfrak{m} \subset
\pi_0 A$, we may assume that $\pi_0 A$ is \emph{complete}; we may do this
because the completion of any noetherian local ring is faithfully flat over it.

Let $k$ 
be the residue field of the (discrete) commutative ring $\pi_0 A$. The \'etale
fundamental group of $\spec \pi_0 A$ is naturally isomorphic to that of $\spec
k$, via the inclusion $\spec k \hookrightarrow \spec \pi_0 A$ as the closed
point, since $\pi_0 A$ is a complete local ring. 
Let $x_1, \dots, x_n \in \pi_0 A$ be generators for the maximal ideal. Consider
the tower of $\e{\infty}$-$A$-algebras
\[\dots \to  A//(x_1^3, \dots, x_n^3) \to A//(x_1^2, \dots, x_n^2) \to A//(x_1,
\dots, x_n) , \]
whose inverse limit is given by $A$ itself (by completeness). 
Observe that the maps at each stage are not uniquely determined. For instance,
to give a map $A//(x_1^2, \dots, x_n^2) \to A//(x_1, \dots, x_n)$ amounts to
giving nullhomotopies of each of $x_1^2, \dots, x_n^2$
in $A//(x_1, \dots, x_n)$, and there are many possible choices of
nullhomotopies. One has to make choices at each stage. 

Denote the $\e{\infty}$-algebras in this tower by $\left\{A_m\right\}$. 
As a result, the equivalence $A \simeq \varprojlim A_m$ leads to a fully
faithful imbedding
\[ \md^\omega(A) \subset \varprojlim \md^\omega(A_m),  \]
from the $\infty$-category $\md^\omega(A)$ of perfect $A$-modules into the
homotopy limit of the $\infty$-categories $\md^\omega(A_m)$ of perfect
$A_m$-modules. As discussed in \cite[\S 7.1]{galgp}, this implies that if we show
that the Galois group of each $A_m$ is equivalent to that of $k$ (i.e., is
algebraic), then the Galois group of $A$ itself is forced to be algebraic.
This, however, is precisely what we proved in \Cref{evenpdicartiniangal} above. 

\end{proof}

\subsection{The general case}

In the previous parts of this section, we showed that the Galois theory of a
rational noetherian $\e{\infty}$-ring $A$ containing a unit in $\pi_2$ was entirely algebraic. In this
subsection, we will explain the modifications needed to handle the case where
we do not have a unit in $\pi_2$; in this case, the structure of the entire
homotopy ring $\pi_* A$ (rather than simply $\pi_0 A$) intervenes. We will
begin with some generalities from \cite{BRalg} which, incidentally, shed further light on
Galois groups of general $\e{\infty}$-rings.

Let $R_*$ be a commutative, $\mathbb{Z}$-graded ring (\emph{not}
graded-commutative!), such as $\pi_{\mathrm{even}}(A)$ for $A \in \clg$. 
We start by setting up a Galois formalism for $R_*$ that takes into account the
grading. 

\begin{definition} 
A \emph{graded finite \'etale} $R_*$-algebra is a commutative, graded
$R_*$-algebra $R'_*$ such that, as underlying commutative rings, the map $R_*
\to R'_*$ is finite \'etale. 
\end{definition} 

We list two fundamental examples: 

\begin{example} 

Given a finite \'etale $R_0$-algebra $R'_0$, then one can build from this a
graded finite \'etale $R_*$-algebra via $R'_* \stackrel{\mathrm{def}}{=} R'_0
\otimes_{R_0} R_*$. 

\end{example} 

\begin{example} 
Let $R_* = \mathbb{Z}[1/n, x_n^{\pm 1}]$ where $|x_n| = n$. Then the map
$R_* \to R'_* =  \mathbb{Z}[1/n, y_1^{\pm 1}], x_n \mapsto y_1^n$ is graded
finite \'etale. 
In other words, one can adjoin $n$th roots of invertible generators in 
degrees divisible by $n$, over a $\mathbb{Z}[1/n]$-algebra. 
\end{example}

Consider the category $\mathcal{C}_{R_*}$ of graded finite \'etale $R_*$-algebras and
graded $R_*$-algebra homomorphisms. 
We start by observing that it is opposite to a Galois category. One can formulate this in
the following manner. The grading on $R_*$ determines an action of the
multiplicative group $\mathbb{G}_m$ on $\spec R_*$, in such a manner that to
give a quasi-coherent sheaf on the quotient stack $(\spec R_*)/\mathbb{G}_m$
is equivalent to giving a graded $R_*$-module. To give a finite \'etale
cover of the quotient stack $(\spec R_*)/\mathbb{G}_m$ is equivalent to giving
a graded $R_*$-algebra which is finite \'etale over $R_*$. In other words,
graded finite \'etale $R_*$-algebras are equivalent to finite \'etale covers
of the \emph{stack} $\spec R_*/\mathbb{G}_m$. 

\begin{definition} 
We define the \emph{graded \'etale fundamental group} $\pi_1^{\mathrm{et, gr}} \spec R_*$ to be the \'etale fundamental
group of the stack $(\spec R_*)/\mathbb{G}_m$. 
\end{definition}

\begin{example} 
\label{grep}
Suppose $R_*$ contains a unit in degree $1$. In this case, $R_* \simeq R_0
\otimes_{\mathbb{Z}} \mathbb{Z}[t^{\pm 1}]$ where $|t| =1 $. In particular, the
quotient stack $(\spec R_*)/\mathbb{G}_m$ is simply $\spec R_0$, so the graded
\'etale fundamental group of $R_*$ is the fundamental group of $\spec R_0$. 
\end{example}

Now, let $A_{\mathrm{even}} = \pi_{\mathrm{even}}(A)$, constructed with the degree
halved (so that $(A_{\mathrm{even}})_1 = \pi_2(A)$) and fix a graded, finite
\'etale $A_{\mathrm{even}}$-algebra $A'_*$. We will construct an
$\e{\infty}$-$A$-algebra $A'$ equipped with an isomorphism
$\pi_{\mathrm{even}}A' \simeq A'_*$ which is a finite cover of $A$,
generalizing the construction that starts with a finite \'etale $\pi_0
A$-algebra and obtains a finite cover of $A$. 
As a result, we will obtain: 

\begin{theorem}[Baker-Richter \cite{BRalg}] 
There is a natural fully faithful imbedding from the category of graded, finite \'etale
$A_{\mathrm{even}}$-algebras into the category of finite covers of $A$  in
$\e{\infty}$-rings. 
\end{theorem} 
Dually, we obtain surjections of profinite groups
\begin{equation} \pi_1 \md(A) \twoheadrightarrow 
\pi_1^{\mathrm{et, gr}}( \spec A_{\mathrm{even}}) \twoheadrightarrow
\pi_1^{\mathrm{et}} (\spec
\pi_0 A),
\end{equation}
refining the surjection \eqref{galsurj}. 
Since our formulation of the (essentially same) result is slightly different 
from that of Baker-Richter, we give the deduction from their statement below. 

\begin{proof} 

Start with a $G$-Galois object $A'_*$ in the category of graded, finite \'etale
$A_{\mathrm{even}}$-algebras (for $G$ a finite group). Then $A'_*$ is a projective
$A_{\mathrm{even}}$-module. Let $\widetilde{A}'_* = A'_*
\otimes_{A_{\mathrm{even}}} \pi_*(A)$ be the graded-commutative algebra
obtained by tensoring up; it acquires a $G$-action. Moreover, it is a
finitely generated, projective $\pi_*(A)$-module, and the map
\[ \widetilde{A}'_* \otimes_{\pi_*(A)} \widetilde{A}'_* \to \prod_G
\widetilde{A}'_*,  \]
given by all the twisted multiplications $a_1 \otimes a_2 \mapsto a_1 g(a_2)$
for $g \in G$, is an isomorphism. By \cite[Theorem 2.1.1]{BRalg}, we can construct an object $A' \in \clg_{A/}$
with a $G$-action such that $A'$ is a $G$-Galois extension of $A$ and realizes
the above map $\pi_*(A) \to \widetilde{A}'_*$ on homotopy groups.

We can now describe the universal property of $A' $ as an
$\e{\infty}$-$A$-algebra. 
In fact, we claim that
for any $\e{\infty}$-$A$-algebra $B$, we have
a natural homotopy equivalence
\begin{equation}\label{univprop}  \hom_{\clg_{A/}}(A', B) = \hom_{A_{\mathrm{even}}}( A'_*,
\pi_{\mathrm{even}}(B)), \end{equation}
so in particular the left-hand-side is discrete. But by Galois descent, this is also 
the $G$-fixed points of the set of maps 
\begin{align*}\hom_{\clg_{A'}}( A' \otimes_A A' , B
\otimes_A A') & \simeq \prod_G \mathrm{Idem}(B \otimes_A A') \\
& = \hom_{A'_*}( A'_* \otimes_{\pi_{\mathrm{even}(A)}} A'_*,
\pi_{\mathrm{even}}(B) \otimes_{\pi_{\mathrm{even}(A)}} A'_*)
,\end{align*}
since $A'$ has homotopy groups which are flat over $\pi_*(A)$ and since $A'_*
\otimes_{\pi_{\mathrm{even}(A)}} A'_* \simeq \prod_G A'_*$. But using the
algebraic form of Galois descent, we get that
$$
\hom_{A'_*}( A'_* \otimes_{\pi_{\mathrm{even}(A)}} A'_*,
\pi_{\mathrm{even}}(B) \otimes_{\pi_{\mathrm{even}(A)}} A'_*)^G = 
 \hom_{A_{\mathrm{even}}}( A'_*,
\pi_{\mathrm{even}}(B))
,$$
so we get \eqref{univprop}
as claimed. 

This imbeds the Galois objects in the category of graded, finite \'etale
$\pi_{\mathrm{even}}(A)$-algebras \emph{fully faithfully} (by \eqref{univprop}) in the category of finite covers of the
$\e{\infty}$-ring. To associate a finite cover to any 
graded, finite \'etale
$\pi_{\mathrm{even}}(A)$-algebra, one now uses Galois descent: the Galois
objects can be used to split any finite \'etale algebra object. Full
faithfulness on these more general covers can now be checked locally, using
descent. 
\end{proof}

With this in mind, we can state and prove our main result. 
\begin{theorem} 
Let $A$ be a noetherian, rational $\e{\infty}$-ring. Then the natural map 
\( \pi_1 \md(A) \to  \pi_1^{\mathrm{et, gr}}( \spec \pi_{\mathrm{even}}(A))  \)
is an isomorphism of profinite group(oid)s. 
\end{theorem} 
\begin{proof} 
We have already proved this result if $A$ has a 
unit in degree two, thanks to \Cref{galep} (see also \Cref{grep}). 
We want to claim that for any $A$ satisfying the conditions of this result, the
functor from graded, finite \'etale $\pi_{\mathrm{even}}(A)$-algebras to finite
covers of the $\e{\infty}$-ring $A$ is an equivalence of categories
$\mathcal{C}_1 \simeq \mathcal{C}_2$. We already know that the functor
$\mathcal{C}_1 \to \mathcal{C}_2$ is fully faithful. 

Both
categories depend functorially on $A$ and have a good theory of descent via the
base-change of $\e{\infty}$-rings $\mathbb{Q} \to \mathbb{Q}[t_2^{\pm 1}]$,
where $\mathbb{Q}[t_2^{\pm 1}]$ is the free rational $\e{\infty}$-ring on an
invertible degree two class. By descent theory, we can thus reduce to the case
of a $\mathbb{Q}[t_2^{\pm 1}]$-algebra, for which we have proved
the result in \Cref{galep}. 
\end{proof}

\section{The Picard group}

\subsection{Generalities}
Another natural invariant that one might attempt to study using the theory of
residue fields is the {Picard group}. Recall that the \emph{Picard
group} of an $\e{\infty}$-ring $A$, 
denoted $\mathrm{Pic}(A)$, is the group of isomorphism classes of
$\otimes$-invertible $A$-modules. 
In fact, these techniques were
originally introduced in \cite{HMS} in the study of the $K(n)$-local Picard
group. 
It is known that if $A$ is an $\e{\infty}$-ring which is even periodic and
whose $\pi_0 $ is regular, then the Picard group is algebraic (\cite{BR}). 
 
In this paper, we have given global descriptions (in terms of algebra) of both
the Galois theory and the thick subcategories 
of $\md(A)$, for $A$ a noetherian rational $\e{\infty}$-ring.
We do not know if it is possible to describe the Picard group of rational
$\e{\infty}$-rings in a similar global manner. However, 
the following example
shows that any answer will be necessarily more complicated. 

\begin{example} 
\label{badpic}
Consider the $\e{\infty}$-ring $A = \sym^* \mathbb{Q}[-1] \otimes
\mathbb{Q}[\epsilon]/\epsilon^2$, which is obtained from the ring of ``dual
numbers'' by freely adding a generator in degree $-1$. 
By the results of \Cref{sec:comparewithloc}, we have
\[ \md(A) \subset \loc_{S^1}(\md( \mathbb{Q}[\epsilon]/\epsilon^2)),  \]
that is, to give an $A$-module is equivalent to giving a
$\mathbb{Q}[\epsilon]/\epsilon^2$-module together with an automorphism whose action on
homotopy groups is ind-unipotent.

For example, we might consider the
$\mathbb{Q}[\epsilon]/\epsilon^2$-module 
$\mathbb{Q}[\epsilon]/\epsilon^2$ and equip it with the automorphism given by
$1 + r \epsilon$, for any $r \in \mathbb{Q}$. 
For any $r \in \mathbb{Q}$, this defines an $A$-module $M_r$. 
Since the underlying 
$\mathbb{Q}[\epsilon]/\epsilon^2$-module is invertible, it follows that $M_r
\in \mathrm{Pic}( \md(A))$. 
Moreover, $M_r \otimes M_{r'} \simeq M_{r+ r'}$ (by composing automorphisms). 
This shows that there is a copy
of $\mathbb{Q}$ inside the Picard group of $\md(A)$ that one does not see from
the homotopy groups of $A$. 
In fact, one sees easily that $\mathrm{Pic}(A) \simeq \mathbb{Z} \oplus
\mathbb{Q}$.
\end{example} 

Nonetheless, we will be able to obtain a weak partial result about the
Picard groups of noetherian rational $\e{\infty}$-rings in comparison to the
algebraic analog. We will first need some algebraic preliminaries. 
\begin{definition}
Given a graded-commutative ring $R_*$, 
the category of graded $R_*$-modules has a symmetric monoidal structure (the
graded tensor product). 
We let $\mathrm{Pic}(R_*)$ denote the Picard group
of the category of graded $R_*$-modules, i.e., the group of
isomorphism classes of invertible graded
$R_*$-modules.
\end{definition}

Note that any invertible graded $R_*$-module must be flat, since tensoring
with it is an autoequivalence (hence exact).

\begin{proposition} 
\label{invertible2periodic}
Suppose $R_*$ is a graded-commutative ring such that $R_2$ contains a unit and
$R_0$ is a noetherian ring, and $R_1$ is a finitely generated
$R_0$-module. Suppose $R_0$ has no nontrivial idempotents.
Then $\mathrm{Pic}(R_*) = \mathbb{Z}/2 \oplus \mathrm{Pic}(R_0)$, where the
$\mathbb{Z}/2$ comes from the shift of $R_*$.
\end{proposition}
\begin{proof}
It suffices to show that if $R_0$ is local, then 
$\mathrm{Pic}(R_*)  = \mathbb{Z}/2$.
We prove this first if $R_0$ is a field $k$, so there is a map $R_* \to
k[t_2^{\pm 1}]$ of graded rings with nilpotent kernel.
Let $M_*$ be an invertible
$R_*$-module. Then $M \otimes_{R_*} k[t_2^{\pm 1}]$ is either $k[t_2^{\pm 1}]$
or its shift since $\mathrm{Pic}(k[t_2^{\pm 1}]) \simeq \mathbb{Z}/2$; assume the former without loss of generality.
Choose a homogeneous $\overline{x} \in (M_* \otimes_{R_*} k[t_2^{\pm 1}])_0$ which is a generator
and lift it to a homogeneous  element $x \in M_0$. 
We then obtain a map $R_* \to M_*$ which induces an isomorphism after
tensoring with $k[t_2^{\pm 1}]$. By Nakayama's lemma, one sees that it
is surjective. Let $K_*$ be the kernel. Then the short exact sequence
\[ 0 \to K_* \to R_* \to M_* \to 0  \]
has the property that
\[  0 \to K_ * \otimes_{R_*}  k[t_2^{\pm 1} ] 
\to k[t_2^{\pm 1}] \to k[t_2^{\pm 1}] \to 0
\]
is still exact, by flatness of $M_*$, and it shows that $K_* = 0$ by
Nakayama's lemma. 

If $R_0$ is not a field $k$, then we can consider the maximal ideal
$\mathfrak{m} \subset R_0$ and consider the map $R_* \to R_*/R_*
\mathfrak{m}$ to reduce to this case. Using a similar argument with Nakayama's lemma, and the fact
that the Picard group of $R_*/R_* \mathfrak{m}$ is $\mathbb{Z}/2$, we find
that the Picard group of $R_*$ is $\mathbb{Z}/2$ as well. 
\end{proof}

Recall that if $A$ is an $\e{\infty}$-ring, one has the following basic
construction. 
\begin{cons}[\cite{BR}]
\label{algpiccons}
There is an inclusion $\mathrm{Pic}(\pi_*(A)) \to \mathrm{Pic}( A)$ which
sends an invertible graded $\pi_*(A)$-module $M_*$ to an invertible $A$-module
$M$
(which is uniquely determined) with $\pi_*(M) \simeq M_*$.
The image consists of those invertible $A$-modules $M$ such that $\pi_*(M)$ is a
\emph{flat} $\pi_*(A)$-module.
\end{cons}

Elements in the image of the map  $\mathrm{Pic}(\pi_*(A)) \to \mathrm{Pic}(A)$ are said to be \emph{algebraic;} if every element is algebraic, then the
Picard group itself is said to be algebraic. 
There are many cases in which the Picard group of an $\e{\infty}$-ring can be
shown to be algebraic. For instance, if $A$ is even periodic with regular
noetherian $\pi_0$, or if $A$ is connective, then it is known \cite{BR} that
the Picard group of $A$ is algebraic (see also \cite[\S 2.4.6]{pictmf}). 
\Cref{badpic} shows that the Picard group of a rational noetherian
$\e{\infty}$-ring need not be algebraic. 
Our main result (\Cref{rationalpic}), however, implies that any \emph{torsion}
in the Picard group is necessarily algebraic.

To prove this result, we will need to use some techniques from \cite{pictmf}
which we review briefly here. Recall (\cite[Def. 2.2.1]{pictmf}) that the
Picard group  $\mathrm{Pic}(A)$ is the group of connected components of a
connective spectrum $\pics(A)$ called the \emph{Picard spectrum} of $A$.
The infinite loop space $\Omega^\infty \pics(A)$ is associated to the
symmetric monoidal $\infty$-groupoid of invertible $A$-modules (under tensor
product). 
The use of $\pics(A)$ (as opposed to simply $\mathrm{Pic}(A)$) is critical when one wishes to appeal to
descent-theoretic techniques.

We now outline a basic descent-theoretic technique in the study of Picard
groups of ring spectra. 
\begin{cons}[{Compare \cite[\S 3]{pictmf}}]
Let $A \to B$ be a morphism of $\e{\infty}$-rings which admits descent. In
this case, we can obtain an expression for the $\infty$-category $\md(A)$ as
the totalization
\[ \md(A) \simeq \mathrm{Tot} \left( \md(B) \rightrightarrows \md(B
\otimes_A B) \triplearrows \dots \right),  \]
by descent theory \cite[\S 3]{galgp}. 
As a result, one obtains an expression for the spectrum $\pics(A)$,
\begin{equation} \label{pictot} \pics(A) = \tau_{\geq 0} \mathrm{Tot}\left(
\pics( B^{(\otimes \bullet + 1)})
\right) , \end{equation}
and a resulting homotopy spectral sequence
\begin{equation} \label{picBKSS}  E_2^{s,t} = H^s\left( \pi_t   \pics(
B^{(\otimes \bullet + 1)})
\right)\implies \pi_{t-s} \pics(A), \quad t-s \geq 0
. \end{equation}
\end{cons}

We recall, moreover, that for any $\e{\infty}$-ring
$A$, we have natural isomorphisms $\pi_1( \pics(A)) \simeq (\pi_0
A)^{\times}$  and $\pi_t( \pics(A)) \simeq \pi_{t-1}(
A)$ for $t \geq 2$. In particular, the descent spectral sequence
\eqref{picBKSS}, for $t \geq 2$, has the same $E_2$-page as the usual Adams
spectral sequence. If $A \to B$ admits descent, it follows from
\cite[Comparison Tool 5.2.4]{pictmf} together with the analogous result for
the $A \to B$ Adams spectral sequence \cite[Cor. 4.4]{galgp} that
\eqref{picBKSS} degenerates (for $t -s \geq 0$) after a finite stage with a
horizontal vanishing line. 

\begin{remark}
We refer to \cite{pictmf} for
several computational examples and applications of this spectral sequence. 
In this paper, we will only use the existence of this spectral sequence and
its degeneration at a finite stage with a horizontal vanishing line. 
\end{remark}

\subsection{The main result}
\begin{theorem} 
\label{rationalpic}
Let $A$ be a rational, noetherian $\e{\infty}$-ring. 
Then the cokernel of the map $\mathrm{Pic}(\pi_*(A)) \to \mathrm{Pic}(A)$  is
torsion-free.
\end{theorem}

The main goal of this subsection 
is to prove \Cref{rationalpic}, which while not entirely satisfying still
provides significant information. In other settings, most of the interesting
information in such Picard groups is precisely the torsion. We will prove this
in several steps, following the construction of residue fields in
\Cref{residueflds}.

\begin{lemma} \label{quotdegminusonepic}
Let $A$ be a rational $\e{\infty}$-ring and let $y \in \pi_{-1}(A)$. 
Then the kernel of $\mathrm{Pic}(A) \to
\mathrm{Pic}(A//y)$ is a $\mathbb{Q}$-vector space.
\end{lemma}
\begin{proof} 
In fact, $A \to A//y$ admits descent (\Cref{quotdegminusone}), so we use the homotopy spectral sequence
associated to the expression 
\eqref{pictot}. 
We observe that, for $i \geq 2$, the cosimplicial abelian group $\pi_i ( \pics( B^{\otimes \bullet + 1}))$
consists of rational vector spaces. Therefore, to prove the lemma, it suffices to
show that there is no contribution in filtration one. 
However, we note that $\pi_0 ( (A//y)^{\otimes k} ) \simeq \pi_0(A//y)[u_1,
\dots, u_{k-1}]$. As a result, the cosimplicial abelian group $\pi_0 (
(A//y)^{\otimes \bullet + 1})^{\times}$ is actually constant and has no nontrivial
cohomology. This proves the claim. 
\end{proof} 

\begin{lemma} 
\label{countablepic}
Let $A$ be a rational $\e{\infty}$-ring such that: 
\begin{enumerate}
\item $\pi_2(A)$ contains a unit. 
\item $\pi_0(A)$ is a field $k$.
\item $\pi_{-1}(A)$ is countably dimensional vector space over $k$.
\end{enumerate}
Then the torsion subgroup of  $\mathrm{Pic}(A)$  is $\left\{A, \Sigma
A\right\}$. 
\end{lemma} 
\begin{proof} We will imitate the argument of \Cref{nilpcountable}.
Consider the sequence of $\e{\infty}$-rings 
$A \simeq A^{(0)} \to A^{(1)} \to A^{(2)} \to \dots $ of
\Cref{countablerational}. By 
\Cref{quotdegminusonepic}, the maps
\[ \mathrm{Pic}(A^{(i)})_{\mathrm{tors}} \to
\mathrm{Pic}(A^{(i+1)})_{\mathrm{tors}}  \]
are injections. 

Let $A_1 = \varinjlim A^{(i)}$. 
It follows that the map $A \to A_1$ induces an injection
$\mathrm{Pic}(A)_{\mathrm{tors}} \to
\mathrm{Pic}(A_1)_{\mathrm{tors}}$, because the Picard functor commutes with
filtered colimits \cite[Prop.
2.4.1]{pictmf}. Moreover, $A_1$ satisfies the same hypotheses, and we can
construct a sequence  of $\e{\infty}$-rings
$A_1 \simeq A_1^{(0)} \to A_1^{(1)} \to \dots $ from \Cref{countablerational}. 
Similarly, each of the maps $A_1^{(i)} \to A_1^{(i+1)}$ induces an injection 
$\mathrm{Pic}(A_1^{(i)})_{\mathrm{tors}} \to
\mathrm{Pic}(A_1^{(i+1)})_{\mathrm{tors}}$. 
If we set $A_2 = \varinjlim A_1^{(i)}$, we get that the map 
$A_1 \to A_2$ induces an injection $\mathrm{Pic}(A_1)_{\mathrm{tors}} \to
\mathrm{Pic}(A_2)_{\mathrm{tors}}$. Inductively, we follow the proof
of \Cref{nilpcountable} and construct the sequence 
\[ A = A_0 \to A_1 \to A_2 \to \dots ,  \]
such that each $A_i$ satisfies the same hypotheses as $A$ did, and such that
$A_i  \to A_{i+1}$ is a colimit of a sequence constructed in
\Cref{countablerational}, 
whose colimit is $k[t_2^{\pm 1}]$.
Our reasoning shows that we get a sequence of injections
\[ \mathrm{Pic}(A)_{\mathrm{tors}} \to \mathrm{Pic}(A_1)_{\mathrm{tors}} \to
\mathrm{Pic}(A_2)_{\mathrm{tors}} \to \dots \to \mathrm{Pic}(k[t_2^{\pm
1}])_{\mathrm{tors}} \simeq \mathbb{Z}/2  \]
because, again, the Picard functor commutes with filtered colimits. This
completes the proof. 
\end{proof}

\begin{lemma} 
\label{modnilpotentinjpic}
Let $A$ be a rational, noetherian $\e{\infty}$-ring. Suppose that $\pi_2(A)$
contains a unit and that $\pi_0(A)$ is a local artinian ring with residue
field. 
Let $x \in \pi_0(A)$ belong to the maximal ideal. Then the map $A \to A//x$ induces an injection 
$\mathrm{Pic}(A)_{\mathrm{tors}} \to \mathrm{Pic}(A//x)_{\mathrm{tors}}$.
\end{lemma} 
\begin{proof} 
The map $A \to A//x$ admits descent since $x$ is nilpotent (\Cref{quotdesc}). 
Therefore, 
we can apply the expression \eqref{pictot} and the associated spectral
sequence. As in \Cref{countablepic}, to run the argument, it suffices to see that there
are no torsion contributions in filtration one. 
All the rings $\pi_0((A//x)^{\otimes (n+1)})$ are local artinian with the same
residue field $k$, in view of \Cref{radicial2}. Given any local artinian ring $R$ with residue field $k$,
the group $R^{\times}$ of units has a natural splitting $k^{\times} \oplus
R^{\times, \mathrm{unip}}$ where $R^{\times, \mathrm{unip}}$ is a
$\mathbb{Q}$-vector space. From this, it follows easily that the contribution
in filtration one in the spectral sequence is a $\mathbb{Q}$-vector space,
which proves the lemma. 
\end{proof}

\begin{lemma} 
\label{torsionart2periodicpic}
Let $A$ be a rational, noetherian $\e{\infty}$-ring. Suppose that $\pi_2(A)$
contains a unit and that $\pi_0(A)$ is a local artinian ring. 
Then the only nontrivial element in the torsion subgroup of $\mathrm{Pic}(A)$
is $\Sigma A$.
\end{lemma} 
\begin{proof} 
As in the proof of \Cref{rsfldforart}, we can construct a sequence of $\e{\infty}$-rings
\[ A  = A_0 \to A_1 \to A_2 \to \dots ,  \]
such that:
\begin{enumerate}
\item  $A_{i+1} \simeq A_i//(x_1^{(i)}, \dots, x_{n_i}^{(i)})$ for some
finite sequence of nilpotent elements $x^{(i)}_{j} \in \pi_0(A_i)$. 
\item Each ring $\pi_0(A_i)$ is local artinian with residue field $k$. 
\item  Each map $A_i \to A_{i+1}$ induces
 a map of local artinian rings $\pi_0(A_i) \to \pi_0(A_{i+1})$ which
 annihilates the maximal ideal of $\pi_0(A_i)$. 
\end{enumerate}
By \Cref{modnilpotentinjpic}, it follows that each of the maps $A_i \to A_{i+1}$ induces an
injection $\mathrm{Pic}(A_i)_{\mathrm{tors}} \to
\mathrm{Pic}(A_{i+1})_{\mathrm{tors}}$. The colimit $A_\infty = \varinjlim_i
A_i$ 
has $\mathrm{Pic}(A_\infty)_{\mathrm{tors}} \simeq \mathbb{Z}/2$ in view of
\Cref{countablepic}, and $\mathrm{Pic}(A)_{\mathrm{tors}} \hookrightarrow
\mathrm{Pic}(A_\infty)_{\mathrm{tors}}$. Putting everything together, we get that
$\mathrm{Pic}(A)_{\mathrm{tors}} \simeq \mathbb{Z}/2$ as desired. 
\end{proof}

\begin{proof}[Proof of \Cref{rationalpic}] 
Let $ M$ be an invertible $A$-module. Suppose that $n > 0$ and that
the tensor power $M^{\otimes n} \in \mathrm{Pic}(A)$ has the property that
$\pi_*(M)$ is a flat $\pi_*(A)$-module, i.e., $M^{\otimes n}$ belongs to the image of 
the map
$\mathrm{Pic}(\pi_*(A)) \to \mathrm{Pic}(A)$. 
We need to show that $M$ itself has this property. 

For this, we may make the base change $A \to
A[t_2^{\pm 1}]$, as $\pi_*(A) \to \pi_*(A[t_2^{\pm 1}])$ is faithfully flat, and thus assume that $\pi_2(A)$ contains a unit. Since
flatness is a local property, it also suffices to work under the assumption
that $\pi_0(A)$ is a local ring. Similarly, by completing, we can assume that $\pi_0(A)
$ is complete local with maximal ideal $\mathfrak{m} = (x_1, \dots, x_n)
\subset \pi_0(A)$. 
Here
we have to use the fact that any invertible $A$-module is necessarily perfect 
\cite[Prop. 2.1.2]{pictmf}. 
In this case, the algebraic Picard group is trivial
(\Cref{invertible2periodic}). So it suffices to show
that $\mathrm{Pic}(A)_{\mathrm{tors}} = \mathbb{Z}/2$. 

Consider the tower of $\e{\infty}$-$A$-algebras $A_m \stackrel{\mathrm{def}}{=}
A//(x_1^m, \dots, x_n^m)$.
We have $\mathrm{Pic}(A_m)_{\mathrm{tors}} \simeq \mathbb{Z}/2$ by
\Cref{torsionart2periodicpic}. 
One sees that $\mathrm{Pic}(A) \subset \pi_0(\varprojlim_m \pics(A_m))$
because $\md^\omega(A) \subset \varprojlim_m \md^\omega(A_m)$.
Note that the tower of abelian groups $\{\pi_1(\pics(A_m))\} =
\{\pi_0(A_m)^{\times}\}$
satisfies the Mittag-Leffler condition since the
tower $\left\{\pi_0(A_m)\right\}$ of finite-dimensional $k$-vector spaces
clearly satisfies the Mittag-Leffler condition. In particular, 
$\pi_0(\varprojlim_m \pics(A_m)) = \varprojlim_m \mathrm{Pic}(A_m)$ by the
Milnor exact sequence.
Finally, 
\[ \mathrm{Pic}(A)_{\mathrm{tors}} \subset \left(\varprojlim_m \mathrm{Pic}(A_m)
\right)_{\mathrm{tors}} \subset \varprojlim_m \mathrm{Pic}(A_m)_{\mathrm{tors}}
= \varprojlim_m \mathbb{Z}/2 = \mathbb{Z}/2.
 \]
\end{proof}

\section{Non-noetherian counterexamples}

The purpose of this section is to describe certain counterexamples that can
arise from non-noetherian rational $\e{\infty}$-rings. In particular, we
obtain as a result new constructions of Galois extensions of ring spectra
(\Cref{strangegalois})
and of elements in Picard groups (\Cref{picstrange}). The main point of this section is that we
can obtain such examples for quasi-affine schemes (such as punctured spectra)
which fail to satisfy a form of ``purity.''

\subsection{Quasi-affineness}

Let $X$ be a noetherian scheme. 
Recall the presentable, stable $\infty$-category $\qcoh(X)$ of quasi-coherent 
sheaves of $\mathcal{O}_X$-complexes on $X$ \cite{DAGQC}. 
$\qcoh(X)$ is a symmetric monoidal $\infty$-category with unit the
structure sheaf $\mathcal{O}_X$, whose endomorphisms are given by the
$\e{\infty}$-ring $\R  \Gamma(X, \mathcal{O} _X)$ of (derived) global sections
of $\mathcal{O}_X$. 
In this generality, we obtain an adjunction
\begin{equation} \label{adjscheme} ( \cdot \otimes_{\R \Gamma(X,
\mathcal{O}_X)} \mathcal{O}_X, \R
\Gamma)\cl \md( \R \Gamma(X, \mathcal{O}_X)) \rightleftarrows \qcoh(X),
\end{equation}
which sends the 
$\R \Gamma(X, \mathcal{O}_X)$-module $\R \Gamma(X, \mathcal{O}_X)$ to the
structure sheaf $\mathcal{O}_X$. The right adjoint takes the derived global
sections.
We need the following basic fact. 

\begin{theorem}[{ \cite[Prop. 2.4.4]{DAGQC}}] \label{quasiaffineness}
Suppose $X$ is quasi-affine. Then the adjunction
\eqref{adjscheme} is a pair of inverse equivalences of symmetric monoidal
$\infty$-categories.
\end{theorem} 

$\e{\infty}$-rings of the form $\R \Gamma(X, \mathcal{O}_X)$ for quasi-affine schemes
will be our primary source of counterexamples, because questions about $\R
\Gamma(X, \mathcal{O}_X)$
can often be reduced to questions (in ordinary algebraic geometry) about the
scheme $X$. 
For instance, we obtain immediately: 

\begin{corollary} 
\label{picqaffine}
Suppose $X$ is a quasi-affine scheme and $\mathcal{L}$ is a line bundle on
$X$. Then the $\R \Gamma(X, \mathcal{O}_X)$-module $\R \Gamma(X, \mathcal{L})$
is invertible.
\end{corollary}

We can also obtain a comparison for Galois theory. 
\begin{corollary} 
\label{galoisqaffine}
Suppose $X$ is a quasi-affine scheme. 
Suppose $Y \to  X$ is a finite \'etale cover. Then the map $\R \Gamma(X,
\mathcal{O}_X) \to \R \Gamma(Y, \mathcal{O}_Y)$ of $\e{\infty}$-rings is a
finite cover. If $Y \to X$ is a $G$-torsor for a finite group $G$, then the
map (together with the natural $G$-action on the target) exhibits 
$\R \Gamma(Y, \mathcal{O}_Y)$ as a faithful $G$-Galois extension of $ \R
\Gamma(X, \mathcal{O}_X)$. In fact, the Galois group of the $\e{\infty}$-ring
$\R \Gamma(X, \mathcal{O}_X)$ is naturally identified with the \'etale
fundamental group of $X$.
\end{corollary} 
\begin{proof} 
If $f\cl Y \to X$ is a finite \'etale cover, then $Y$ can be constructed as the
relative  $\mathrm{Spec}$ of a sheaf of commutative $\mathcal{O}_X$-algebras
$f_* (\mathcal{O}_Y)$. On any affine open $\spec R \subset X$,
$f_*(\mathcal{O}_Y)|_{\spec R}$ is obtained from a finite \'etale algebra. 
In other words, if we write $\qcoh(X)  = \varprojlim_{\spec R \subset X}
\md(R)$, as the inverse limit ranges over all Zariski open subsets $\spec R
\subset X$, then $f_*(\mathcal{O}_Y) \in \qcoh(X)$ defines a family of finite
covers in each of these symmetric monoidal $\infty$-categories. It follows
from \cite[Prop. 7.1]{galgp} 
(and \cite[Th. 6.5]{galgp} for the comparison between weak finite covers and
finite covers)
that $f_*(\mathcal{O}_Y)$ defines a finite cover
of the unit in the $\infty$-category $\qcoh(X)$. Thus, applying the
equivalence $\R \Gamma$ then completes the proof. 
\end{proof}

In general, there is no reason for the Galois theory (resp. Picard group) of
the quasi-affine scheme to be determined by that of $\pi_0 \R \Gamma(X,
\mathcal{O}_X)$, and this will be a source of counterexamples.
This is related to subtle purity questions. However, we can obtain
immediately counterexamples to the thick subcategory theorem without
noetherian hypotheses in view of the
following result. 

\begin{theorem}[{Thomason \cite[Th. 3.15]{thomason}}]
Let $X$ be a noetherian scheme. Let $\qcoh^\omega(X)$ denote the
$\infty$-category of 
quasi-coherent complexes $\mathcal{F} \in \qcoh(X)$ such that for every open
affine $\spec R \subset X$, $\mathcal{F}( \spec R)$ is a perfect $R$-module
(equivalently, $\qcoh^\omega(X)$ consists of the dualizable objects in
$\qcoh(X)$).
Then the thick tensor-ideals in $\qcoh^\omega(X)$ are in natural bijection
with the specialization-closed subsets of $X$.
\end{theorem} 

Suppose $X$ is quasi-affine. In this case, $\qcoh^\omega(X)$ 
corresponds under the equivalence of \Cref{quasiaffineness} to the $\R \Gamma(X, \mathcal{O}_X)$-modules which are dualizable,
i.e., perfect, and thick tensor-ideals are the same as thick subcategories. 
Thus, we get:
\begin{corollary} 
\label{thicksubcatquasiaffine}
If $X$ is a quasi-affine scheme, then the thick 
subcategories of $\md^\omega( \R \Gamma(X, \mathcal{O}_X))$ are in natural
bijection  with the specialization-closed subsets of $X$. 
\end{corollary} 

\begin{cons}
Let $X$ be a quasi-affine, noetherian scheme which is not affine. 
In this case, the natural map $X \to \spec \pi_0 (\R \Gamma(X, \mathcal{O}_X))$
is an open immersion 
of schemes \cite[Tag 01P5]{stacks-project}. 
By \Cref{thicksubcatquasiaffine}, the thick subcategories of 
$\md( \R \Gamma(X, \mathcal{O}_X))$ are in bijection not with
specialization-closed subsets of the scheme $\spec \pi_0 \left(\R \Gamma(X,
\mathcal{O}_X)\right)$, but rather specialization-closed subsets of an open
subset (i.e., $X$) of this scheme.
\end{cons}

\subsection{The punctured affine plane}
\label{bigcompact}
Not every compact $\e{\infty}$-$\mathbb{Q}$-algebra has the noetherianness
properties used 
in this paper. 
In this subsection, we explain how $\e{\infty}$-rings of the form $\R \Gamma(X,
\mathcal{O}_X)$ can give counterexamples, and work out the simplest
nontrivial case. 

\begin{cons}
Consider the $\e{\infty}$-ring $A$ of functions on the
punctured affine
plane $\mathbb{A}^2 \setminus \{(0, 0)\}$, which fits into a homotopy pullback
\[ \xymatrix{
A \ar[d] \ar[r] & \mathbb{Q}[x^{\pm 1}, y] \ar[d] \\
\mathbb{Q}[x, y^{\pm 1}] \ar[r] &  \mathbb{Q}[x^{\pm 1}, y^{\pm 1}]
}.\]
The homotopy groups $\pi_*(A)$ are given by 
\[ 
\pi_i(A) = \begin{cases} 
\mathbb{Q}[x,y] & i =  0\\
\mathbb{Q}[x, y]/(x^\infty, y^\infty) & i = - 1 \\
0 & \text{otherwise}
 \end{cases}, 
\]
where $\mathbb{Q}[x,y]/(x^\infty, y^\infty)$ denotes 
the cokernel of the map $\mathbb{Q}[x^{\pm 1}, y] \oplus \mathbb{Q}[x, y^{\pm
1}] \to \mathbb{Q}[x^{\pm 1}, y^{\pm 1}]$. 
In particular, $\pi_{-1}(A)$ is not a finitely generated $\pi_0(A)$-module
(though $\pi_0(A)$ is noetherian). 
\end{cons}

\begin{proposition}
The $\e{\infty}$-ring $A$ is compact in $\clg_{\mathbb{Q}/}$. 
\end{proposition}We are grateful to J. Lurie
for explaining this to us. 
\begin{proof} 

In fact, to give a morphism $A \to B$,
for $B$ a rational $\e{\infty}$-ring, is equivalent to giving two elements $u,
v \in \Omega^\infty B$ which have the \emph{property} that $B/(x,y) $ is
contractible. This follows from the fact that $A$ is the \emph{finite
localization} (\cite{miller}) of $\mathbb{Q}[x, y]$ away from the $\mathbb{Q}[x,y]$-module
$\mathbb{Q}[x,y]/(x,y)$ (which is supported at the origin). 
In particular, to give a morphism of $\e{\infty}$-rings $A \to B$ is equivalent
to giving a map $\mathbb{Q}[x, y] \to B$ \emph{such that} $B//(x,y) = B/(x,y) $ is
contractible; note that this condition is detected in a finite stage of a
filtered colimit. 
\end{proof}
In \cite[Ex.~2.8]{BHL}, B. Bhatt and D. Halpern-Leistner give in fact an \emph{explicit
presentation} of $A$ as an $\e{\infty}$-ring under $\mathbb{Q}[x,y]$. 
Consider the $\mathbb{Q}[x,y]$-module $M = \mathbb{Q}[x,y]/(x,y)$ and the
natural map $\phi\colon \mathbb{Q}[x,y] \to M$. The dual gives a map $\psi \colon
\mathbb{D}M \to \mathbb{Q}[x,y]$, where $\mathbb{D}M$ is the Spanier-Whitehead
dual of $M$. Then one has:
\begin{proposition}[Bhatt, Halpern-Leinster]
The $\e{\infty}$-$\mathbb{Q}[x,y]$-algebra is the pushout
\[ \xymatrix{
\sym_{\mathbb{Q}[x,y]}^*( \mathbb{D} M) \ar[d] \ar[r] & \mathbb{Q}[x,y] \ar[d]
\\
\mathbb{Q}[x,y] \ar[r] &  A
},\]
where:
\begin{enumerate}
\item $\sym_{\mathbb{Q}[x,y]} ^*(\mathbb{D}M)$ is the free
$\e{\infty}$-$\mathbb{Q}[x,y]$-algebra on the $\mathbb{Q}[x,y]$-module
$\mathbb{D}M$.
\item The two maps $\sym_{\mathbb{Q}[x,y]}^*( \mathbb{D}M) \to
\mathbb{Q}[x,y]$ are adjoint to two maps of
$\mathbb{Q}[x,y]$-modules $\mathbb{D}M \to \mathbb{Q}[x,y]$ which are 
given by $\psi$ and the zero map. 
\end{enumerate}\end{proposition} 
\begin{proof} 
Indeed, 
let $A'
\in \clg_{\mathbb{Q}[x,y]/}$.
Then 
\[ \hom_{\clg_{\mathbb{Q}[x,y]/}}(A, A') \simeq 
\ast \times_{\hom_{\md({\mathbb{Q}[x,y])}}( \mathbb{D}M, A')} \ast
.\]
Here the two maps $\ast \to 
\hom_{\md({\mathbb{Q}[x,y])}}( \mathbb{D}M, A')$ 
send, respectively, $\ast $ to $0$ and to 
the map $\mathbb{D}M \stackrel{\psi}{\to} A \to A'$.
If $A'//(x,y) = 0$, then $\mathbb{D}M = 0$ and the mapping space
$\hom_{\clg_{\mathbb{Q}[x,y]/}}(A,A') $ is therefore
contractible. If not, then $\mathbb{D} M \to A \to A'$ is not the zero map, so
that mapping space is empty. This is precisely the universal property of $A'
\in \clg_{A'/}$.
\end{proof} 

\begin{remark}
The $\infty$-category of $A$-modules is equivalent to the $\infty$-category of
quasi-coherent sheaves on the scheme $\mathbb{A}^2_{\mathbb{Q}} \setminus \left\{(0,
0)\right\}$, since this scheme is quasi-affine. In particular, it follows
\Cref{thicksubcatquasiaffine} that the thick
subcategories of $\md^\omega(A)$ correspond to the subsets 
of $\mathbb{A}^2_{\mathbb{Q}} \setminus \left\{(0,
0)\right\}$ which are closed under specialization. 
In particular, \Cref{thicksubcat} fails for $A$, as there is no thick
subcategory corresponding to the origin in $\spec \pi_0 A$. 
\end{remark}

\begin{remark} 
$A$ also illustrates the failure of \Cref{radicial2} in the non-noetherian
case. In fact, $\pi_0(A)/(x,y) \simeq \mathbb{Q}$ while $A//(x,y)$ is
contractible.  
\end{remark} 
\subsection{Punctured spectra and counterexamples}

We will now describe counterexamples to our theorems on Galois groups and
Picard groups in the non-noetherian case, arising from quasi-affine schemes in a
similar way. We start by recalling 
the context. 

\begin{cons}
Let $(R, \mathfrak{m})$ be a noetherian local ring. We define the
\emph{punctured spectrum} $\psp R = \spec R \setminus
\left\{\mathfrak{m}\right\}$.
\end{cons}

The punctured spectrum $\psp R$ is a quasi-affine scheme, and many
``purity'' results in 
algebraic geometry and commutative algebra relate invariants of $\psp R$ to
those of 
$\spec R$.
\begin{theorem}[{Zariski-Nagata \cite[Exp. X, Th. 3.4]{SGA2}}] \label{ZN} Let $R$ be a regular local ring of dimension
$\geq 2$. Then the inclusion $\psp R \to \spec R $ induces an isomorphism on
\'etale fundamental groups. 
\end{theorem} 

From our point of view, we can restate ``purity'' results such as the
Zariski-Nagata theorem in terms of ring spectra, by passage to the $\e{\infty}$-ring $\R\Gamma( \psp R,
\mathcal{O}_{\psp R})$. 
Let $(R, \mathfrak{m})$ be a regular local ring of dimension $\geq 2$.
Then $\pi_0(\R\Gamma(\psp R,
\mathcal{O}_{\psp R})) = R$ and $\pi_{-i} \R\Gamma(\psp R,
\mathcal{O}_{\psp R})) = 0$ if $i \notin \left\{0, \mathrm{dim}(R) -
1\right\}$ by general results on local cohomology and depth \cite[Exp. III,
Ex. 3.4]{SGA2}. 
For example, we obtain:

\begin{enumerate}
\item \Cref{ZN}  is thus equivalent to the statement that the Galois group of the
$\e{\infty}$-ring  
$\R\Gamma(\psp R,
\mathcal{O}_{\psp R})
$ is algebraic.
\item Similarly, on a much more elementary level, let $(R, \mathfrak{m})$ be a
regular local ring of dimension $\geq 2$. Then $R$ is factorial, so that it
has trivial Picard group. Since the Picard group is isomorphic to the class
group, it follows that the inclusion $\psp R \to \spec R $ induces an
isomorphism on Picard groups. In particular, it follows that the 
Picard group of $\R\Gamma( \psp R, \mathcal{O}_{\psp R})$ is algebraic. (More
subtle purity results of the Picard group in non-regular cases can be phrased
in this form too.) 
\end{enumerate}

The main point of this subsection is that 
non-regular rings for which purity fails can be used to give interesting
examples of Galois extensions and invertible modules over non-noetherian ring
spectra. Our example (which is not local) follows \cite[Example 16.5]{divKrull}.

Let $K$ be a field of characteristic zero containing a primitive $n$th root
$\zeta_n$ of unity. Let $m \geq 2$. 
Consider the $\mathbb{Z}/n$-action on the ring $R' = K[x_1, \dots, x_m]$
sending $x_i \mapsto \zeta_n x_i$. Then $R = R'^{\mathbb{Z}/n} $ is the
subring generated by all the homogeneous degree $n$ monomials.
Geometrically, the map 
\( \spec R' \to \spec R = \spec R'^{\mathbb{Z}/n}  \)
corresponds to the quotient of the affine space $\mathbb{A}^m$ by 
rotation by the angle $2\pi/n$ in each direction. In particular, this map is \'etale away from the origin,
the only place where the action fails to be free. 

\begin{cons}
Let $X = \spec R$
and let $Y = \spec R'$. We have a $\mathbb{Z}/n$-action on $Y$ and a map $Y \to
X$ which exhibits $X$ as the quotient $Y/(\mathbb{Z}/n)$. 
If $y \in Y$ is the point corresponding to the prime ideal $(x_1, \dots,
x_m)$ and $x \in X$ its image, then the point $y$ is $\mathbb{Z}/n$-invariant,
and the 
induced
map $Y \setminus \left\{y\right\} \to X \setminus \left\{x\right\}$
is a $\mathbb{Z}/n$-torsor. 
We write $X^{\circ} = X \setminus \left\{x\right\}, Y^{\circ} = Y \setminus
\left\{y\right\}$.

We define $\e{\infty}$-rings $A = \R \Gamma( X^{\circ},
\mathcal{O}_{X^{\circ}})$ and $B = \R
\Gamma(Y^{\circ}, \mathcal{O}_{Y^{\circ}})$.
Note that $B \in \clg_{A/}$ has a natural $\mathbb{Z}/n$-action.
\end{cons}

\begin{theorem} 
\label{strangegalois}
We have $\pi_0(A) = R, \pi_0(B) = R'$. The map $A \to B$, together with the
$\mathbb{Z}/n$-action on $B$ exhibits $B$ as a faithful $\mathbb{Z}/n$-Galois
extension of $A$.
\end{theorem} 
\begin{proof} 
The natural map $R' = \Gamma(Y, \mathcal{O}_Y) \to \Gamma(Y^{\circ},
\mathcal{O}_{Y^{\circ}})$ is an isomorphism since $Y$ is normal and the
missing locus is codimension $\geq 2$ by \cite[\S 17, Th. 35]{matsumura}. 
Moreover, the $\mathbb{Z}/n$-torsor $Y^{\circ} \to X^{\circ}$ shows that 
$\R\Gamma(X^{\circ}, \mathcal{O}_{X^{\circ}}) \simeq 
\R\Gamma(Y^{\circ}, \mathcal{O}_{Y^{\circ}})^{h\mathbb{Z}/n}$. Taking $\pi_0$,
we find that $\pi_0(A) \simeq 
\left(\pi_0 \left( \Gamma(Y^{\circ},
\mathcal{O}_{Y^{\circ}})\right)\right)^{\mathbb{Z}/n} = R'^{\mathbb{Z}/n}  = R$.
The assertion that $A \to B$ is a faithful $\mathbb{Z}/n$-Galois extension
comes from \Cref{galoisqaffine}.
\end{proof}

The map of commutative rings $R \to R'$ is not \'etale: in fact, $R'$ is not
regular (at zero) while $R$ is. 
In particular, the Galois extension of 
\Cref{strangegalois} does not come from algebra. 
\begin{example} 
Suppose $K = \mathbb{C}$ is the field of complex numbers. In this case, the
topological realization of $\spec
R$ (i.e., the topological space $\mathbb{C}^m/(\mathbb{Z}/n)$) is easily seen to be \emph{contractible}: one can 
scale down to the image of the origin. Therefore, the \'etale fundamental
group of $R = \pi_0(A)$ is trivial. However, the $\mathbb{Z}/n$-Galois
extension $A \to B$ (and the fact that $B$ has trivial Galois group by
\Cref{galoisqaffine} applied to $\mathbb{C}^m \setminus \left\{(0,
\dots, 0)\right\}$) shows that the Galois group of the $\e{\infty}$-ring $A$
is \emph{precisely} $\mathbb{Z}/n$. 
\end{example}

We can also obtain elements of the Picard group. 

\begin{example} 
\label{picstrange}
We compute the (classical) Picard group of the scheme $X^{\circ}$ again with
$K = \mathbb{C}$. First, we observe that
the Picard group of $Y^{\circ}$ is trivial since that of affine space $Y$ is
and $Y^{\circ} = Y \setminus \left\{y\right\}$.  
Moreover, $\Gamma(Y^{\circ}, \mathcal{O}_{Y}^{\circ}) = \mathbb{C}^{\times}$. 
We have a $\mathbb{Z}/n$-torsor $Y^{\circ} \to X^{\circ}$, and we
can use 
Galois descent to compute the Picard group as
$\mathrm{Pic}(X^{\circ}) = H^1(\mathbb{Z}/n; H^0(
Y^{\circ}, \mathcal{O}_{Y^{\circ}}^{\times})) = \mathbb{Z}/n$ since the
action is trivial.
Therefore, the Picard group of the $\e{\infty}$-ring $A$ is given by
$\mathbb{Z} \oplus \mathbb{Z}/n$ where the $\mathbb{Z}$ comes from suspensions. 
However, we claim that the Picard group of $\pi_0(A)$ is trivial. 
Indeed, by \Cref{normalinj} below, we have $\mathrm{Pic}(X) \subset
\mathrm{Pic}(X^{\circ}) = \mathbb{Z}/n$. But the Picard group of $X$ can have
no torsion as $X$ is topologically contractible, in view of the K\"ummer sequence, and
therefore $\mathrm{Pic}(X) = 0$.
In particular, by \Cref{picqaffine}, we find that $\mathrm{Pic}(A) \simeq \mathbb{Z}
\oplus \mathbb{Z}/n$ though the Picard group of $\pi_0(A)$ is trivial. 
\end{example} 

\begin{lemma} 
\label{normalinj}
Let $X$ be a noetherian, normal, integral scheme. Let $Z \subset X$ have
codimension $\geq 2$. Then the map $\mathrm{Pic}(X) \to \mathrm{Pic}(X
\setminus Z)$ is injective.
\end{lemma} 
\begin{proof} 
Let $j\cl X \setminus Z \to X$ be the open imbedding. Then the map
$\mathcal{O}_X^{\times} \to j_*(\mathcal{O}_{X
\setminus Z}^{\times}) $ is an isomorphism. The Leray spectral sequence now
shows that the natural map $H^1(X, \mathcal{O}_X^{\times}) \to H^1( X
\setminus Z, \mathcal{O}_{X \setminus Z}^{\times})$ is an injection. 
\end{proof}

\bibliographystyle{alpha}
\bibliography{rational}

\end{document}